\documentclass[12pt, a4paper]{amsart}
\usepackage{amscd, amssymb,amsmath,upref, color}

\usepackage[all]{xy}
\usepackage{amscd}
\allowdisplaybreaks
\begin{document}
\newcommand{\Chi}{\raisebox{2pt}{\ensuremath{\chi}}}

\def\C{\mathbb{C}}
\def\N{\mathbb{N}}
\def\Z{\mathbb{Z}}
\def\R{\mathbb{R}}
\def\T{\mathbb{T}}

\def\H{\mathcal{H}}
\def\B{\mathcal{B}}
\def\F{\mathcal{F}}
\def\K{\mathcal{K}}
\def\TT{\mathcal{T}}
\def\QQ{\mathcal{Q}}
\def\AA{\mathcal{A}}
\def\EE{\mathcal{E}}
\def\L{\mathcal{L}}
\def\O{\mathcal{O}}
\def\OO{\mathcal{O}}

\def\supp{\operatorname{supp}}
\def\Aut{\operatorname{Aut}}
\def\End{\operatorname{End}}
\def\Ad{\operatorname{{Ad}}}
\def\Tr{\operatorname{Tr}}
\def\id{\operatorname{id}}
\def\lsp{\operatorname{span}}
\def\UHF{\operatorname{UHF}}
\def\tr{\operatorname{tr}}
\def\clsp{\operatorname{\overline{span}}}
\def\To{\mathcal{T}}

\newcommand{\under}{\backslash}
\newcommand{\qregular}{regular }

\renewcommand{\theenumi}{\alph{enumi}} 
\renewcommand{\theenumii}{\roman{enumii}}

\newtheorem{thm}{Theorem}  [section]
\newtheorem{cor}[thm]{Corollary}
\newtheorem{lemma}[thm]{Lemma}
\newtheorem{prop}[thm]{Proposition}
\theoremstyle{definition}
\newtheorem{defn}[thm]{Definition}
\newtheorem{remark}[thm]{Remark}
\newtheorem{example}[thm]{Example}
\newtheorem{keyexample}[thm]{Key Example}
\newtheorem{remarks}[thm]{Remarks}

\numberwithin{equation}{section}

\title[Stacey crossed products associated to Exel systems]
{\boldmath{Stacey crossed products associated to Exel systems}}

\author[an Huef]{Astrid an Huef}
\author[Raeburn]{Iain Raeburn}
\address{Department of Mathematics and Statistics\\ University of Otago\\ PO Box 56\\Dunedin 9054\\ New Zealand}
\email{astrid@maths.otago.ac.nz, iraeburn@maths.otago.ac.nz}

\begin{abstract}
There are many different crossed products by an endomorphism of a $C^*$-algebra, and constructions by Exel and Stacey have proved particularly useful. Here we show that every Exel crossed product is isomorphic to a Stacey crossed product, though by a different endomorphism of a different $C^*$-algebra. We apply this result to a variety of Exel systems, including those associated to shifts on the path spaces of directed graphs.
\end{abstract}

\thanks{Some of the results in this paper, including most of those in Sections~\ref{sec:main} and \ref{sec:graph}, were contained in the preliminary version \cite{aHRexel}, which we posted on the \texttt{arXiv} in July 2011.} 
\date{2 November 2011}
\maketitle

\section{Introduction}
Everybody agrees that when $\alpha$ is an automorphism of a unital $C^*$-algebra $A$, the crossed product $A\rtimes_\alpha\Z$ is generated by a unitary element $u$ and a representation $\pi:A\to A\rtimes_{\alpha}\Z$ which satisfy the covariance relation
\begin{equation}\label{covauto}
\pi(\alpha(a))=u\pi(a)u^*\ \text{ for $a\in A$.}
\end{equation}
The covariance relation can be reformulated as $\pi(\alpha(a))u=u\pi(a)$ or $u^*\pi(\alpha(a))u=\pi(a)$, and, when $\alpha$ is an automorphism, these reformulations are equivalent to \eqref{covauto}. When $\alpha$ is an endomorphism, though, these reformulations are no longer equivalent, and give different crossed products. Thus there are several crossed products based on covariance relations in which $u$ is an isometry \cite{Mur, S, E}, and still more crossed products in which $u$ is a partial isometry \cite{Lin,ABL}. The crossed products of Stacey \cite{S} and Exel \cite{E} have proved to be particularly useful, and have been extensively studied, for example in \cite{P, BKR, Mur2} and \cite{EV, BR, EaHR}.

Stacey's crossed product $B\times_\alpha\N$ is generated by an isometry $s$ and a representation $\pi$ of $A$ satisfying $\pi(\alpha(a))=s\pi(a)s^*$. His motivating example was the endomorphism $\alpha$ of the UHF core $A$ in the Cuntz algebra $\OO_n$ described by Cuntz in \cite{C}, for which we recover $\OO_n$ as $A\times_\alpha \N$, and is especially useful for \emph{corner endomorphisms} which map $A$ onto the corner $\alpha(1)A\alpha(1)$ (see also  \cite{P, BKR}). Stacey's construction has been extended to semigroups of endomorphisms, and these semigroup crossed products were used to study Toeplitz algebras \cite{ALNR,LR1}; they have since been used extensively in the analysis of $C^*$-algebras arising in number theory (see \cite{LR2, Bren, L, Li, LNT}, for example).

Exel's construction depends on the choice of a transfer operator $L:A\to A$ for $\alpha$, which is a positive linear map satisfying $L(\alpha(a)b)=aL(b)$. He uses $L$ to build a Hilbert bimodule $M_L$ over $A$, and then his crossed product $A\rtimes_{\alpha,L}\N$ is closely related to the Cuntz-Pimsner algebra $\OO(M_L)$ of this bimodule (the precise relationship is described in \cite{BR}). The motivating example for Exel's construction is the endomorphism of $C(\{1,\cdots,n\}^\infty)$ induced by the backward shift, for which averaging over the $n$ preimages of each point gives a transfer operator $L$ such that $\OO_n\cong C(\{1,\cdots,n\}^\infty)\rtimes_{\alpha,L}\N$. More generally, Exel realised each Cuntz-Krieger algebra  as a crossed product by the corresponding subshift of finite type, thereby giving a very direct proof that the Cuntz-Krieger algebra is determined up to isomorphism by the subshift. Exel's construction has attracted a good deal of attention in connection with irreversible dynamics \cite{EV, EaHR, CS}, and has also been extended to semigroups of endomorphisms with interesting consequences \cite{Lar, BaHLR}.

Here we study a large family of Exel systems $(A,\alpha,L)$, and describe a general procedure which builds a corner endomorphism $\beta$ of another $C^*$-algebra $B$ such that Exel's $A\rtimes_{\alpha,L}\N$ is naturally isomorphic to Stacey's $B\times_\beta \N$. In our motivating example $(C(\{1,\cdots,n\}^\infty),\alpha,L)$, we recover Cuntz's description of $\O_n$ as the Stacey crossed product by an endomorphism of the UHF core (see Remark~\ref{remCuntz}). Our investigations of this construction have led us to some interesting new examples of Exel and Stacey crossed products associated with graph algebras and Cuntz algebras.

We begin with short sections in which we describe the two kinds of crossed products of interest to us, and a family of relative Cuntz-Pimsner algebras associated to Exel systems. Then in \S\ref{sec:main}, we give the details of our construction of corner endomorphisms, first in the generality of relative Cuntz-Pimsner algebras, then specialising to Exel crossed products and Cuntz-Pimsner algebras. Then we derive a six-term exact sequence in $K$-theory from a result of Paschke \cite{P2} about Stacey crossed products, and compare our construction to a previous one of Exel \cite{E2}. 

In \S\ref{sec:graph}, we consider a directed graph $E$ and the Exel system $(C^*(E^\infty),\alpha,L)$ introduced in \cite{BRV}, and obtain a new description of the graph algebra $C^*(E)$ as a Stacey crossed product of the core (which generalises results of R\o rdam and Kwa\'sniewski for finite graphs \cite{Ror,Kwa}). This led us to revisit Exel systems involving other endomorphisms of the UHF core in Cuntz algebras, which we do in \S\ref{sec:Cuntz}. In \S\ref{sec:graph} we needed a detailed analysis of the core in the $C^*$-algebra $C^*(E)$ of a column-finite graph, and we describe this analysis in an appendix.



\section{Exel systems}

An endomorphism $\alpha$ of a $C^*$-algebra $A$ is \emph{extendible} if it extends to a strictly continuous endomorphism $\overline{\alpha}$ of $M(A)$. Nondegenerate endomorphisms, for example, are automatically extendible with $\overline{\alpha}(1)=1$. In this paper we are interested in \emph{Exel systems} $(A,\alpha,L)$ of the kind studied in \cite{BRV}, which means that $\alpha$ is an extendible endomorphism of a $C^*$-algebra $A$  and  $L:A\to A$ is a  positive linear map which  extends to a  positive linear map $\overline{L}:M(A)\to M(A)$ such that 
\begin{equation}\label{transferoponmofa}
L(\alpha(a)m)=a\overline{L}(m)\ \text{ for $a\in A$ and $m\in M(A)$.}
\end{equation}
Equation~\eqref{transferoponmofa} implies that $\overline{L}$ is strictly continuous, and we then assume further that $\overline{L}(1_{M(A)})=1_{M(A)}$ (but not that $\overline{\alpha}$ is unital).  Following Exel \cite{E}, we say that $L$ is a \emph{transfer operator} for $(A, \alpha)$.

We write $M_L$ for the Hilbert bimodule over $A$ used in \cite{BRV}, where the construction  of \cite{E,BR} was extended to non-unital $A$. Briefly,  $A$ is given a bimodule structure by $a\cdot m=am$ and $m\cdot b=m\alpha(b)$, and the pairing $\langle m\,,\, n\rangle= L(m^*n)$ defines a pre-inner product on $A$.  Modding out by the elements of norm $0$ and completing gives a right-Hilbert $A$--$A$ bimodule $M_L$ (or \emph{correspondence} over $A$). We denote by $q:A\to M_L$ the canonical map of $A$ onto a dense sub-bimodule of $M_L$, and by $\phi$ the homomorphism of $A$ into $\L(M_L)$ implementing the left action. 

We are interested in several $C^*$-algebras associated to Exel systems, and more specifically to the bimodule $M_L$.

\subsection{Relative Cuntz-Pimsner algebras} Following \cite{FR}, a representation $(\psi,\pi)$ of $M_L$ in a $C^*$-algebra $B$  consists of a linear map $\psi:M_L\to B$ and a homomorphism $\pi:A\to B$ such that
\[
\psi(m\cdot a)=\psi(m)\pi(a),\quad \psi(m)^*\psi(n)=\pi(\langle m\,,\, n\rangle), \text{\ and\ }\psi(a\cdot m)=\pi(a)\psi(m)
\]
for $a\in A$ and $m,n\in M_L$. By \cite[Proposition~1.8]{FR}, a representation $(\psi,\pi)$ of $M_L$ gives a representation $(\psi^{\otimes i},\pi)$ of $M_L^{\otimes i}:=M_L\otimes_A\cdots\otimes_A M_L$ such that $\psi^{\otimes i}(m_1\otimes_A\cdots\otimes_A m_i)=\psi(m_1)\cdots\psi(m_i)$. The Toeplitz algebra $\TT(M_L)$  is the $C^*$-algebra generated by a universal  representation $(j_M,j_A)$ of $M_L$, and \cite[Lemma~2.4]{FR} says that
\[
\TT(M_L)=\clsp\big\{j_M^{\otimes i}(m)j_M^{\otimes j}(n)^*:\;i,j\in\N,\;m\in M_L^{\otimes i},\;n\in M_L^{\otimes j}\big\}.
\]
There is  a strongly continuous \emph{gauge action} $\gamma:\T\to\Aut \TT(M_L)$ such that $\gamma_z(j_A(a))=j_A(a)$ and $\gamma_z(j_M(m))=zj_M(m)$ \cite[Proposition~1.3]{FR}.

A representation $(\psi,\pi)$ of $M_L$ in $B$ gives  a homomorphism $(\psi,\pi)^{(1)}:\K(M_L)\to B$ such that $(\psi,\pi)^{(1)}(\Theta_{m,n})=\psi(m)\psi(n)^*$ for every rank-one operator $\Theta_{m,n}:p\mapsto m\cdot\langle n\,,\, p\rangle$. Following \cite[\S1]{FMR}, if $J$ is an ideal of $A$ contained in $\phi^{-1}(\K(M_L))$, then we view the \emph{relative Cuntz-Pimsner algebra} $\O(J,M_L)$ of \cite{MS} as the quotient of $\TT(M_L)$ by the ideal generated by
\[
\big\{j_A(a)-(j_M, j_A)^{(1)}(\phi(a)): a\in J\big\}.
\]
We write $Q$ or $Q_J$ for the quotient map.  Then $(k_M,k_A):=(Q_J\circ j_M,Q_J\circ j_A)$ is universal for  representations $(\psi,\pi)$ which are \emph{coisometric} on $J$ (that is, satisfy $\pi|_J= (\psi,\pi)^{(1)}\circ\phi|_J$).  If $J=\{0\}$ then $\O(J,M_L)$ is just $\TT(M_L)$; if $J=\phi^{-1}(\K(M_L))$ then $\O(J,M_L)$ is the \emph{Cuntz-Pimsner algebra} $\O(M_L)$ of \cite{Pim}, and representations that are coisometric on $\phi^{-1}(\K(M_L))$ are called \emph{Cuntz-Pimsner covariant}.

It follows from \cite[Lemma~2.4]{FR} that each $\O(J,M_L)$ carries a gauge action $\gamma$ of $\T$ such that $Q_J$ is equivariant, and  the fixed-point algebra or \emph{core}  is
\[
\O(J,M_L)^\gamma=\clsp\big\{k_M^{\otimes i}(m)k_M^{\otimes i}(n)^*:m,n\in M_L^{\otimes i},\;i\in\N\big\}.
\]

\subsection{Exel's crossed product} As in \cite{BRV}, a \emph{Toeplitz-covariant representation} of $(A,\alpha,L)$ in a $C^*$-algebra $B$ consists of a nondegenerate homomorphism $\pi:A\to B$ and an element $S\in M(B)$ such that
\[S\pi(a)=\pi(\alpha(a))S\quad\text{and}\quad
S^*\pi(a)S=\pi(L(a)).
\]
The \emph{Toeplitz crossed product} $\TT(A,\alpha,L)$ is generated by a universal Toeplitz-covariant representation $(i,s)$ (or, more correctly, by $i(A)\cup i(A)s$).  By \cite[Proposition~3.1]{BRV}, there is a map $\psi_s:M_L\to \TT(A,\alpha,L)$ such that $\psi_s(q(a))=i(a)s$, and $(\psi_s,i)$ is a  representation of $M_L$. Set 
\[
K_\alpha:=\overline{A\alpha(A)A}\cap \phi^{-1}(\K(M_L)).
\]
 In \cite[\S4]{BRV}, the \emph{Exel crossed product} $A\rtimes_{\alpha,L}\N$ of a possibly non-unital $C^*$-algebra $A$ by $\N$ is the quotient of $\TT(A,\alpha,L)$ by the ideal generated by
\[
\big\{i(a)-(\psi_s,i)^{(1)}(\phi(a)): a\in K_\alpha\big\}.
\]
We write $\QQ$ for the quotient map of $\TT(A,\alpha,L)$ onto $A\rtimes_{\alpha,L}\N$, and $(j,t):=(\QQ\circ i,\overline{\QQ}(s))$. There is a \emph{dual action} $\hat\alpha$ of $\T$ on $A\rtimes_{\alpha,L}\N$ such that $\hat\alpha_z(j(a))=j(a)$ and $\hat\alpha_z(t)=zt$.

Theorem~4.1 of \cite{BRV} says that there is an isomorphism $\theta$ of $\O(K_\alpha,M_L)$ onto $A\rtimes_{\alpha,L}\N$ such that $\theta\circ k_A=\QQ\circ i$ and $\theta\circ k_M=\QQ\circ \psi_s$.

\section{Stacey crossed products}

Suppose that $\alpha $ is an endomorphism of a $C^*$-algebra $A$. A \emph{Stacey-covariant representation} of $(A,\alpha)$ in a $C^*$-algebra $B$ consists of a nondegenerate homomorphism $\pi:A\to B$ and an isometry $V\in M(B)$ such that $\pi(\alpha(a))=V\pi(a)V^*$. Stacey showed in \cite[\S3]{S} that there is a crossed product $A\times_\alpha\N$ which is generated by a universal Stacey-covariant representation $(i_A,v)$. If $(\pi,V)$ is a Stacey-covariant representation of $(A,\alpha)$ in $B$, then we write $\pi\times V$ for the nondegenerate homomorphism of $A\times_\alpha \N$ into $B$ such that $(\pi\times V)\circ i_A=\pi$ and $(\pi\times V)(v)=V$. (Stacey called $A\times_\alpha\N$  ``the multiplicity-one crossed product" of $(A,\alpha)$.) 

The crossed product $A\times_\alpha\N$ carries a dual action $\hat\alpha$ of $\T$, which is characterised by $\hat\alpha_z(i_A(a))=i_A(a)$ and $\hat\alpha_z(v)=zv$. The following ``dual-invariant uniqueness theorem'' says that this dual action identifies $A\times_\alpha\N$ among $C^*$-algebras generated by Stacey-covariant representations of $(A,\alpha)$. It was basically proved in \cite[Proposition~2.1]{BKR} (modulo the correction made in \cite{ALNR}).  

\begin{prop}\label{diut}
Suppose that $\alpha$ is an endomorphism of a $C^*$-algebra $A$, and $(\pi,V)$ is a Stacey-covariant representation of $(A,\alpha)$ in a $C^*$-algebra $D$. If $\pi$ is faithful and there is a strongly continuous action $\gamma:\T\to \Aut D$ such that $\gamma_z(\pi(a))=\pi(a)$ and $\gamma_z(V)=zV$, then $\pi\times V$ is faithful on $A\times_\alpha\N$.
\end{prop}

\begin{proof}
The conditions on $\gamma$ say that $(\pi\times V)(\hat\alpha_z(b))=\gamma_z((\pi\times V)(b))$ for all $b\in A\times_\alpha\N$. Thus
\begin{align*}
\Big\|(\pi\times V)\Big(\int_{\T}\hat\alpha_z(b)\,dz\Big)\Big\|&=\Big\|\int_{\T}(\pi\times V)(\hat\alpha_z(b))\,dz\Big\|=\Big\|\int_{\T}\gamma_z((\pi\times V)(b))\,dz\Big\|\\
&\leq \int_{\T}\|\gamma_z((\pi\times V)(b))\|\,dz=\int_{\T}\|(\pi\times V)(b)\|\,dz\\
&=\|(\pi\times V)(b)\|,
\end{align*}
We now take $(B,\beta)=(A\times_\alpha\N,\hat\alpha)$, and apply \cite[Lemma~2.2]{BKR} to $(B,\beta)$. We have just verified the hypothesis (2) of \cite[Lemma~2.2]{BKR}. The other hypothesis (1) asks for $\pi\times V$ to be faithful on the fixed-point algebra $B^\beta=(A\times_\alpha\N)^{\hat\alpha}$. However, the proof of \cite[Lemma~1.5]{ALNR} uses neither that $A$ is unital nor the estimate (ii) in \cite[Theorem~1.2]{ALNR}, and hence we can deduce from that proof that $\pi\times V$ is faithful on $B^\beta$. Thus \cite[Lemma~2.2]{BKR} applies, and the result follows.
\end{proof}

\section{Exel crossed products as Stacey crossed products}\label{sec:main}

The goal of this section is the following theorem and some corollaries.

\begin{thm}\label{relCPalg=St}
Suppose that $(A,\alpha,L)$ is an Exel system, and $J$ is an ideal of $A$ contained in $\phi^{-1}(\K(M_L))$.  
 There is a unique isometry $V\in M(\O(J,M_L))$ such that $k_M(q(a))=k_A(a)V$ for $a\in A$, and $\Ad V$ restricts to an endomorphism $\alpha'$ of the core $C_J:=\O(J,M_L)^\gamma$ such that 
\begin{equation}\label{defalpha'}
\alpha'\big(k_M^{\otimes i}(a\cdot m)k_M^{\otimes i}(b\cdot n)^*\big)
=k_M^{\otimes (i+1)}(q(\alpha(a))\otimes_Am)k_M^{\otimes (i+1)}(q(\alpha(b))\otimes_An)^* 
\end{equation}
for $a,b\in A$ and $m,n\in M_L^{\otimes i}$.
Further,  $\alpha'$ is extendible with $\overline{\alpha'}(1)=VV^*$, is injective and has range $\overline{\alpha'}(1)C_J\overline{\alpha'}(1)$.
Finally,  $(\id,V)$ is a Stacey-covariant representation of $(C_J,\alpha')$ such that $\id\times V$ is an isomorphism of the Stacey crossed product $C_J\times_{\alpha'}\N$ onto $\O(J,M_L)$.
\end{thm}

\begin{proof}
Since we know from \cite[Corollary~3.5]{BRV} that $k_A:A\to \O(J,M_L)$ is nondegenerate\footnote{We caution that this nondegeneracy is not at all obvious, and even slightly surprising, because the representation $\pi$ in a Toeplitz representation $(\psi,\pi)$ is not required to be nondegenerate.}, there is at most one multiplier $V$ satisfying $k_M(q(a))=k_A(a)V$, and we have uniqueness.

When the $C^*$-algebra $A$ has an identity, we can deduce from the results in \cite[\S3]{BR} that $V:=k_M(q(1))$ has the required properties. When $A$ does not have an identity, we take an approximate identity $\{e_{\lambda}\}$ for $A$, and claim, following Fowler \cite[\S3]{F}, that $\{k_M(q(e_{\lambda}))\}$ converges strictly in $M(\O(J,M_L))$ to a multiplier $V$. Indeed, for $a,b\in A$, $m\in M_L^{\otimes i}$ and $n\in M_L^{\otimes j}$ we have
\begin{align}\label{psialcvgeleft}
k_M(q(e_\lambda))k_M^{\otimes i}(a\cdot m)k_M^{\otimes j}(b\cdot n)^* 
&=k_M(q(e_\lambda))k_A(a)k_M^{\otimes i}(m)k_M^{\otimes j}(b\cdot n)^*\\
&=k_M(q(e_\lambda\alpha(a)))k_M^{\otimes i}(m)k_M^{\otimes j}(b\cdot n)^*\notag\\
&\to k_M(q(\alpha(a)))k_M^{\otimes i}(m)k_M^{\otimes j}(b\cdot n)^*,\notag
\end{align}
and similarly
\begin{equation}\label{psialcvgeright}
k_M^{\otimes i}(a\cdot m)k_M^{\otimes j}(b\cdot n)^*k_M(q(e_\lambda))\to k_M^{\otimes i}(a\cdot m)k_M^{\otimes j}(n)^*k_M(q(b^*)).
\end{equation}
Since $\overline{L}$ is positive and $\overline{L}(1)=1$, we have $\|L\|\leq 1$, and $\|q(e_{\lambda})\|\leq \|e_{\lambda}\|\leq 1$ for all $\lambda$. Thus, since the elements $k_M^{\otimes i}(a\cdot m)k_M^{\otimes j}(b\cdot n)^*$ span a dense subspace of $\O(J,M_L)$, an $\epsilon/3$ argument using \eqref{psialcvgeleft} and~\eqref{psialcvgeright} shows that $\{k_M(q(e_{\lambda}))b\}$ and $\{bk_M(q(e_{\lambda}))\}$ converge in $\O(J,M_L)$ for every $b\in \O(J,M_L)$. Thus $\{k_M(q(e_{\lambda}))\}$ is strictly Cauchy, and since $M(\O(J,M_L)$ is strictly complete we deduce that $\{k_M(q(e_{\lambda}))\}$ converges strictly to a multiplier $V$; \eqref{psialcvgeleft} implies that $V$ satisfies
\begin{equation}\label{propV}
Vk_M^{\otimes i}(a\cdot m)k_M^{\otimes j}(b\cdot n)^*=k_M(q(\alpha(a)))k_M^{\otimes i}(m)k_M^{\otimes j}(b\cdot n)^*.
\end{equation}

To see that $V$ is an isometry, we observe that 
\[
k_M(q(e_{\lambda}))^*k_M(q(e_{\lambda}))=k_A(\langle q(e_{\lambda}),q(e_{\lambda})\rangle)=k_A(L(e_{\lambda}^2));
\]
since $\overline{L}$ is strictly continuous, $L(e_{\lambda}^2)$ converges strictly to $\overline{L}(1_{M(A)})=1_{M(A)}$.  Since $k_A:A\to \O(J,M_L)$ is nondegenerate, $k_M(q(e_{\lambda}))^*k_M(q(e_{\lambda}))$ converges strictly to $1=1_{M(\O(J,M_L))}$, and since the multiplication in a multiplier algebra is jointly strictly continuous on bounded sets (by another $\epsilon/3$ argument), we deduce that $V^*V=1$. Thus $V$ is an isometry. Next, we let $a\in A$ and compute
\begin{equation}
k_M(q(a))=\lim_{\lambda}k_M(q(ae_\lambda))=\lim_{\lambda}k_M(a\cdot q(e_\lambda))=\lim_{\lambda}k_A(a)k_M(q(e_\lambda))=k_A(a)V,
\end{equation}
so $V$ has the required properties. 

Conjugating by the isometry $V\in M(\O(J,M_L))$ gives an endomorphism $\Ad V:T\mapsto VTV^*$, and two applications of \eqref{propV} show that
\begin{align}\label{formforalpha'}
\Ad V\big(k_M^{\otimes i}(a\cdot m)k_M^{\otimes i}(b\cdot n)^*\big)
&=k_M(q(\alpha(a)))k_M^{\otimes i}(m)k_M^{\otimes i}(n)^*k_M(q(\alpha(b)))^*\\
&=k_M^{\otimes(i+1)}(q(\alpha(a))\otimes_A m)k_M^{\otimes (i+1)}(q(\alpha(b))\otimes_An)^*.\notag
\end{align}
The formula \eqref{formforalpha'} implies that $\Ad V$ maps the core $C_J$ into itself, and that the restriction $\alpha':=(\Ad V)|_{C_J}$ satisfies \eqref{defalpha'}. The pair $(\id,V)$ is then by definition Stacey covariant for $\alpha'\in \End C_J$, and we can apply the dual-invariant uniqueness theorem (Proposition~\ref{diut}) to the gauge action $\gamma$ on $\O(J,M_L)$, finding that $\id\times V$ is a faithful representation of $C_J\times_{\alpha'}\N$ in $\O(J,M_L)$. The identity $k_M(q(a))=k_A(a)V$ implies that $k_M(M_L)=\overline{k_A(A)V}$, and since $\O(J,M_L)$ is generated by $k_A(A)\subset C_J$ and $k_M(M_L)$, the range of $\id\times V$ contains the generating set $k_A(A)\cup k_A(A)V$, and hence is all of $\O(J,M_L)$.

It remains for us to prove the assertions about $\alpha'$.  It is injective because $\Ad V$ is. Since  $k_A$ is nondegenerate the image $\{k_A(e_\lambda)\}$ of an approximate identity $\{e_\lambda\}$ converges strictly to $1$ in $M(\O(J,M_L))$. Hence $\{\alpha'(k_A(e_\lambda))\}=\{Vk_A(e_\lambda) V^*\}$ converges strictly to the projection $VV^*$. Thus $\alpha'$ is extendible with $\overline{\alpha'}(1)=VV^*$ by, for example, \cite[Proposition~3.1.1]{adji}. 

The range of $\alpha'$ is certainly contained in the corner $\overline{\alpha'}(1)C_J \overline{\alpha'}(1)$. To see the reverse inclusion, fix $T\in \overline{\alpha'}(1)C_J\overline{\alpha'}(1)$. Then $T=VV^*SVV^*=\Ad V(V^*SV)$ for some $S\in C_J$. 
Since $C_J$ is $\alpha'$-invariant, to see that $T$ is in the range of $\alpha'=\Ad V|_{C_J}$  it suffices to see that $V^*SV$ is in $C_J$, and, by continuity of $\Ad V^*$, it suffices to see this for $S$ of the form
\[
S=k_M^{\otimes i}((a\cdot m_1)\otimes_Am')k_M^{\otimes i}((b\cdot n_1)\otimes_A n')^*
\]
where $a, b\in A$, $m_1, n_1\in M_L$ and $m',n'\in M_L^{\otimes (i-1)}$. For $i\geq 1$ the calculation
\begin{align*}
V^*k_M^{\otimes i}((a\cdot m_1)\otimes_Am')
&=(k_A(a^*)V)^*k_M(m_1)k_M^{\otimes (i-1)}(m')\\
&=k_M(q(a^*))^*k_M(m_1)k_M^{\otimes (i-1)}(m')\\
&=k_A(\langle q(a^*),m_1\rangle)k_M^{\otimes (i-1)}(m')\\
&=k_M^{\otimes (i-1)}(\langle q(a^*),m_1\rangle\cdot m')
\end{align*}
gives 
\[
V^*SV=k_M^{\otimes (i-1)}(\langle q(a^*),m_1\rangle\cdot m')k_M^{\otimes (i-1)}(\langle q(b^*),n_1\rangle\cdot n')^*\in C_J.
\]
For $i=0$  we have 
\[V^*SV=V^*k_A(a)k_A(b)^*V=k_M(q(a^*))^*k_M(q(b^*))=k_A(\langle q(a^*)\,,\, q(b^*)\rangle)\in C_J.
\]
Thus $T=\Ad V(V^*SV)=\alpha'(V^*SV)$, and $\alpha'$ has range $\overline{\alpha'}(1)C_J \overline{\alpha'}(1)$. 
\end{proof}

\subsection{Exel crossed products}
We will now use the isomorphism $\theta$ of $\O(K_\alpha,M_L)$ onto $A\rtimes_{\alpha,L}\N$ to transfer the conclusions of Theorem~\ref{relCPalg=St} over to $A\rtimes_{\alpha,L}\N$.

If $\{e_\lambda\}$ is an approximate identity for $A$, then
\[
\theta\circ k_M(q(e_\lambda))=\QQ(\psi_s(q(e_\lambda)))=\QQ(i(e_\lambda)s)=\QQ(i(e_{\lambda}))\overline{\QQ}(s)
\]
converges by nondegeneracy of $i$ to $\overline{\QQ}(s)$, and hence $\overline{\theta}$ carries the isometry $V$ of Theorem~\ref{relCPalg=St} into $t:=\overline{\QQ}(s)$.  
The isomorphism $\theta$ is equivariant for the gauge action $\gamma$ on $\O(K_\alpha,M_L)$ and the dual action $\hat\alpha$ on $A\rtimes_{\alpha, L}\N$, and hence maps the core $C_{K_\alpha}=\O(K_\alpha,M_L)^\gamma$ onto $(A\rtimes_{\alpha,L}\N)^{\hat\alpha}$. 

Next, we want a workable description of $(A\rtimes_{\alpha,L}\N)^{\hat\alpha}$. 
For $m=q(a_1)\otimes_A\cdots\otimes_Aq(a_i)\in M_L^{\otimes i}$ we have
\begin{align}\label{simplifyinEcp}
k_M^{\otimes i}(m)&=k_M(q(a_1))\cdots\ k_M(q(a_i))=k_A(a_1)V\cdots k_A(a_i)V\\
&=k_A(a_1)k_A(\alpha(a_2))V^2k_A(a_3)V\cdots k_A(a_i)V\notag\\
&=k_A\big(a_1\alpha(a_2)\alpha^{2}(a_3)\cdots \alpha^{i-1}(a_i)\big)V^i,\notag
\end{align}
and hence $\theta$ takes $k_M^{\otimes i}(m)$ into an element of the form $j(a)t^i$. Now
\[
\lsp\big\{j(a)t^it^{*k}j(b):a,b\in A,\;i,k\in\N\big\}
\]
is a $*$-subalgebra of $A\rtimes_{\alpha,L}\N$ containing the generating set $j(A)\cup j(A)t$, and hence is dense in $A\rtimes_{\alpha,L}\N$. The expectation onto $(A\rtimes_{\alpha,L}\N)^{\hat\alpha}$ is continuous and kills terms with $i\not= k$, so 
\begin{equation}\label{presentfpas}
(A\rtimes_{\alpha,L}\N)^{\hat\alpha}=\clsp\big\{j(a)t^it^{*i}j(b):a,b\in A,\;i\in\N\big\}.
\end{equation}

The next corollary says that every Exel crossed product is a Stacey crossed product.

\begin{cor}\label{Ex=St2}
Suppose that $(A,\alpha,L)$ is an Exel system. Then there is an injective endomorphism $\beta$ of $(A\rtimes_{\alpha,L}\N)^{\hat\alpha}$ such that
\begin{equation}\label{formforbeta}
\beta(j(a)t^it^{*i}j(b))=j(\alpha(a))t^{i+1}t^{*(i+1)}j(\alpha(b)).
\end{equation}
The endomorphism $\beta$ is extendible with $\overline\beta(1)=tt^*$ and has range $tt^*(A\rtimes_{\alpha,L}\N)^{\hat\alpha}tt^*$. The pair $(\id,t)$ is a Stacey-covariant representation of $\big((A\rtimes_{\alpha,L}\N)^{\hat\alpha},\beta\big)$, and $\id\times t$ is an isomorphism of the Stacey crossed product $(A\rtimes_{\alpha,L}\N)^{\hat\alpha}\times_\beta\N$ onto the Exel crossed product $A\rtimes_{\alpha,L}\N$.
\end{cor}

\begin{proof}
Applying the isomorphism $\theta:\O(K_\alpha, M_L)\to A\rtimes_{\alpha, L}\N$ to the conclusion of Theorem~\ref{relCPalg=St} gives an endomorphism $\beta:=\theta\circ\alpha'\circ\theta^{-1}$ of $(A\rtimes_{\alpha,L}\N)^{\hat\alpha}$ and an isomorphism $\id\times t$ of $(A\rtimes_{\alpha,L}\N)^{\hat\alpha}\times_\beta\N$ onto $A\rtimes_{\alpha,L}\N$. It remains for us to check the formula for $\beta$. Let $m=q(a_1)\otimes_A\cdots\otimes_Aq(a_i)$. Then the calculation \eqref{simplifyinEcp} shows that $\theta$ carries $k_M^{\otimes i}(c\cdot m)$ into $j(ca_1\alpha(a_2)\cdots\alpha^{i-1}(a_i))t^i$, and $k_M^{\otimes(i+1)}(q(\alpha(c))\otimes_Am)$ into 
\[
j\big(\alpha(c)\alpha(a_1)\alpha^2(a_2)\cdots\alpha^i(a_i)\big)t^{i+1}=j\big(\alpha(ca_1\alpha(a_2)\cdots\alpha^{i-1}(a_i))\big)t^{i+1},
\]
so \eqref{formforbeta} follows from \eqref{defalpha'}.
\end{proof}

\begin{remark}\label{idcpCP}
Since we have identified how the action $\alpha'$ on the core of  $\O(K_\alpha, M_L)$ pulls over to the Exel crossed product $A\rtimes_{\alpha, L}\N$, we will from now on freely identify $\O(K_\alpha, M_L)$ and $A\rtimes_{\alpha, L}\N$, and drop the isomorphism $\theta$ from our notation.
\end{remark}

\begin{example}
We now discuss a family of Exel systems studied in \cite{EaHR} and \cite{LRR}. Let $d\in \N$ and fix $B\in M_d(\Z)$ with nonzero determinant $N$. (This matrix $B$ plays the same role as the matrix $B$ in \cite{EaHR,LRR}; because $A$ is already heavily subscribed in this paper, we write $B^t$ in place of the matrix $A$ used there.) We choose a set $\Sigma$ of coset representatives for $\Z^d/B\Z^d$.

The map $\sigma_{B^t}:\T^d\to \T^d$ characterised by $\sigma_{B^t}(e^{2\pi ix})=e^{2\pi iB^tx}$ is an $N$-sheeted covering map, and induces an endomorphism $\alpha_{B^t}:f\mapsto f\circ\sigma_{B^t}$ of $C(\T^d)$, and $L(f)(z):=N^{-1}\sum_{\sigma_{B^t}(w)=z}f(w)$ defines a transfer operator $L$ for $\alpha_{B^t}$. Proposition~3.3 of \cite{LRR} says that the Exel crossed product $C(\T^d)\rtimes_{\alpha_{B^t},L}\N$ is the universal $C^*$-algebra generated by a unitary representation $u$ of $\Z^d$ and an isometry $v$ satisfying
\begin{itemize}
\item[(E1)] $vu_m=u_{Bm}v$ for $m\in\Z^d$, 
\smallskip
\item[(E2)] $v^*u_mv=\begin{cases} u_{B^{-1}m}&\text{if $m\in B\Z^d$}\\
0&\text{otherwise, and}\end{cases}$
\smallskip
\item[(E3)] $1=\sum_{m\in\Sigma} (u_mv)(u_mv)^*$;
\end{itemize}
we then have
\[
C(\T^d)\rtimes_{\alpha_{B^t},L}\N=\clsp\{u_mv^kv^{*l}u_n^*:k,l\in \N\text{ and }m,n\in \Z^d\}.
\]
When we view $C(\T^d)\rtimes_{\alpha_{B^t},L}\N$ as a Cuntz-Pimsner algebra $(\O(M_L),j_{M_L},j_{C(\T^d)})$ as in Remark~\ref{idcpCP}, $v=j_{M_L}(q(1))=t$ and $u$ is the representation $m\mapsto j_{C(\T^d)}(\chi_m)$, where $\chi_m(z):=z^m$. Since $\alpha_{B^t}(\chi_m)=\chi_{Bm}$, Corollary~\ref{Ex=St2} gives an endomorphism $\beta$ of
\[
\big(C(\T^d)\rtimes_{\alpha_{B^t},L}\N\big)^{\hat\alpha_{B^t}}=\clsp\{u_mv^iv^{*i}u_n^*:i\in \N\text{ and }m,n\in \Z^d\}
\]
satisfying
\begin{equation}\label{defbeta4}
\beta(u_mv^iv^{*i}u_n^*)=u_{Bm}v^{i+1}v^{*(i+1)}u_{Bn}^*.
\end{equation}

Now that we have the formula for $\beta$, we can prove directly that there is such an endomorphism. To see this, recall from \cite[Proposition~5.5(b)]{LRR} that for each $i$ and 
\[
\Sigma_i:=\{\mu_1+B\mu_2+\cdots B^{i-1}\mu_i:\mu\in\Sigma^i\},
\]
$\{u_mv^iv^{*i}u_n^*:m,n\in\Sigma_i\}$ is a set of nonzero matrix units. Thus there is a homomorphism $\zeta_i:M_{\Sigma_i}(\C)\to M_{\Sigma_{i+1}}(\C)$ such that $\zeta_i(u_mv^iv^{*i}u_n^*)=u_{Bm}v^{i+1}v^{*(i+1)}u_{Bn}^*$. Let $\delta$ denote the universal representation of $B^{i+1}\Z^d$ in $C^*(B^{i+1}\Z^d)$. Then the unitary representation $B^im\mapsto \delta_{B^{i+1}m}$ induces a homomorphism $\eta_i:C^*(B^i\Z^d)\to C^*(B^{i+1}\Z^d)$. When we identify $C_i:=\clsp\{u_mv^iv^{*i}u_n^*\}$ with $M_{\Sigma_i}(\C)\otimes C^*(B^i\Z^d)$ as in \cite[Proposition~5.5(c)]{LRR}, we get homomorphisms 
\[
\beta_i:=\zeta_i\otimes \eta_i:C_i=M_{\Sigma_i}(\C)\otimes C^*(B^i\Z^d)\to C_{i+1}=M_{\Sigma_{i+1}}(\C)\otimes C^*(B^{i+1}\Z^d) 
\]
satisfying \eqref{defbeta4}. The Cuntz relation (E3) implies that $\beta_{i+1}|_{C_i}=\beta_i$, and hence the $\beta_i$ combine to give a homomorphism $\beta:\bigcup_{i=1}^\infty C_i\to \bigcup_{i=1}^\infty C_i$ satisfying \eqref{defbeta4}; since the homomorphisms $\beta_i$ are norm-decreasing, $\beta$ extends to an endomorphism of $\O(M_L)^\gamma=\overline{\bigcup_{i=1}^\infty C_i}$.
\end{example}

\subsection{Cuntz-Pimsner algebras}\label{cp-algs}  Let $(A, \alpha, L)$ be an Exel system, and let $I$ be an ideal in $A$. Following \cite[Definition~4.1]{BR}, we say that $L$  is \emph{almost faithful on $I$} if 
\[a\in I\text{\ and\ }L(b^*a^*ab)=0\text{\ for all $b\in A$ imply }a=0. 
\]
Since 
\begin{equation}\label{almostfaithful}
L(b^*a^*ab)=\langle q(ab)\,,\, q(ab)\rangle=\langle\phi(a)(q(b))\, ,\,\phi(a)(q(b))\rangle,
\end{equation}
$L$ is almost faithful on $I$ if and only if $\phi|_I:I\to \L(M_L)$ is injective. Almost faithful transfer operators were introduced in \cite{BR} as a necessary and sufficient condition for the the canonical map $k_A:A\to \O(M_L)$ to be injective (see also \cite[Proposition~2.1]{MS} and \cite[Theorem~4.3]{BRV}).

In this subsection, we suppose that $L$ is almost faithful on $\phi^{-1}(\K(M_L))$. Then, since $k_A:A\to \O(M_L)$ is injective, we can use \cite[Corollary~4.9]{FMR} to realise the core $\O(M_L)^\gamma$ as a direct limit, and obtain a different description of the Stacey system $(\O(M_L)^\gamma,\alpha')$. 

We denote the identity operator on $M_L^{\otimes i}$ by $1_i$, and write $\K(M_L^{\otimes j})\otimes_A 1_{i-j}$ for the image of $\K(M_L^{\otimes j})$ under the map $T\mapsto T\otimes_A 1_{i-j}$. Then, following \cite[\S4]{FMR}, we define
\[
C_i=(A\otimes_A1_i)+(\K(M_L)\otimes_A1_{i-1})+\cdots+\K(M_L^{\otimes i}),
\]
which is a $C^*$-subalgebra of $\L(M_L^{\otimes i})$. We define $\phi_i:C_i\to C_{i+1}$ by $\phi_i(T)=T\otimes_A 1_1$, and define $(C_\infty,\iota^i):=\varinjlim (C_i,\phi_i)$. Since $k_A$ is injective, we can now  apply \cite[Corollary~4.9]{FMR} to the Cuntz-Pimsner covariant representation $(k_M,k_A)$, and deduce that there is an isomorphism $\kappa$ of $C_\infty$ onto the core $\O(M_L)^\gamma$ such that
\begin{equation}\label{defkappa}
\kappa(\iota^i(T\otimes_A1_{i-j}))=(k_M^{\otimes j},k_A)^{(1)}(T)\ \text{ for $T\in \K(M_L^{\otimes j})$ and $i\geq j$}
\end{equation}
(the notation in \cite{FMR} suppresses the maps $\iota^i$ ).

To describe the endomorphism $\beta:=\kappa^{-1}\circ\alpha'\circ\kappa$ of $C_\infty$, we need some notation.

\begin{lemma}\label{defU}
The map $U:A\to M_L$ defined by $U(a)=q(\alpha(a))$ is an adjointable isometry such that $U^*(q(a))=L(a)$ for $a\in A$.
\end{lemma}

\begin{proof}
The calculation
\begin{align*}
\langle U(a)\,,\,U(b)\rangle&=\langle q(\alpha(a))\,,\,q(\alpha(b))\rangle=L(\alpha(a)^*\alpha(b))\\
&=L(\alpha(a^*b))=a^*b=\langle a\,,\,b\rangle
\end{align*}
shows that  $U$ is inner-product preserving. We next note that 
\begin{equation}\label{compU*}
\langle U(a)\,,\,q(b)\rangle=L(\alpha(a)^*b)=a^*L(b)=\langle a\,,\,L(b)\rangle.
\end{equation}
Equation~\eqref{compU*} implies that 
\[
\|a^*L(b)\|=\|\langle U(a)\,,\,q(b)\rangle\|\leq \|U(a)\|\,\|q(b)\|=\|a\|\,\|q(b)\|;
\]
thus $\|L(b)\|\leq\|q(b)\|$, and there is a well-defined bounded linear map $T:M_L\to A$ such that $T(q(b))=L(b)$. Now \eqref{compU*}
shows that $U$ is adjointable with adjoint $T$.
\end{proof}

\begin{cor}\label{defUi}
Define maps $U_i:M_{L}^{\otimes i}\to M_L^{\otimes(i+1)}$ by identifying $M_{L}^{\otimes i}$ with $A\otimes_A M_L^{\otimes i}$ and taking
\[
U_i:=U\otimes_A 1_i:M^{\otimes i}=A\otimes_A M_L^{\otimes i}\to M_L\otimes_A M_L^{\otimes i}=M_L^{\otimes(i+1)}.
\]
Then each $U_i$ is an adjointable isometry, and
\begin{enumerate}
\setlength\itemindent{-16pt} 
\item $U_i(a\cdot m)=q(\alpha(a))\otimes_A m$;
\item $U_i^*(q(a)\otimes_Am)=L(a)\cdot m$;
\item $U_{i+1}=U_i\otimes_A1$ and $U_{i+1}^*=U_i^*\otimes_A 1$.
\end{enumerate}
\end{cor}

With the notation of Corollary~\ref{defUi}, we can now describe  the endomorphism on $C_\infty$.  We reconcile  our results  with Exel's \cite[Theorem~6.5]{E2} in \S\ref{connectExel} below, and show there that Corollary~\ref{CP=St} reduces to  \cite[Theorem~6.5]{E2} when $A$ is unital.

\begin{cor}\label{CP=St}
Suppose that $(A,\alpha,L)$ is an Exel system such that  $L$ is almost faithful on $\phi^{-1}(\K(M_L))$. There is an endomorphism $\beta$ of $C_\infty:=\varinjlim(C_i,\phi_i)$ such that 
\begin{equation}
\beta(\iota^i(T))=\iota^{i+1}(U_iTU_i^*)\ \text{ for $T\in C_i$,}\label{defbeta2}
\end{equation}
and $\beta$ is extendible and injective with range $\overline{\beta}(1)C_\infty\overline{\beta}(1)$. Let $\kappa$ be the isomorphism of $C_\infty$ onto $\O(M_L)^\gamma$ satisfying \eqref{defkappa}, and let $V$ be the isometry in $M(\O(M_L))$ such that $k_M(q(a))=k_A(a)V$ for $a\in A$ (as given by Theorem~\ref{relCPalg=St}). Then $(\kappa,V)$ is a Stacey-covariant representation of $(C_\infty,\beta)$ in $\O(M_L)$, and $\kappa\times V$ is an isomorphism of the Stacey crossed product $C_\infty\times_\beta\N$ onto $\O(M_L)$. If $A$ acts by compact operators on $M_L$, then $\overline{\beta}(1)$ is a full projection.
\end{cor}

\begin{proof}
We define $\beta:=\kappa^{-1}\circ\alpha'\circ \kappa$, where $\alpha'$ is the endomorphism from Theorem~\ref{relCPalg=St}. Let $a\cdot m,b\cdot n\in M_L^{\otimes i}$. Then $\kappa\circ \iota^i(\Theta_{a\cdot m,b\cdot n})=k_M^{\otimes i}(a\cdot m)k_M^{\otimes i}(b\cdot n)^*$, and  
\begin{align*}
\kappa\circ\beta\circ\iota^i(\Theta_{a\cdot m,b\cdot n})&=\alpha'\circ\kappa\circ\iota^i(\Theta_{a\cdot m,b\cdot n})\\
&=k_M^{\otimes (i+1)}(q(\alpha(a))\otimes_Am)k_M^{\otimes (i+1)}(q(\alpha(b))\otimes_An)^*\quad\quad\text{(using \eqref{defalpha'})}\\
&=\kappa\circ\iota^{i+1}(\Theta_{q(\alpha(a))\otimes_Am,q(\alpha(b))\otimes_A n})\\
&=\kappa\circ\iota^{i+1}(\Theta_{U_i(a\cdot m),U_i(b\cdot n)})\\
&=\kappa\circ\iota^{i+1}(U_i\Theta_{a\cdot m,b\cdot n}U_i^*).
\end{align*}
This gives \eqref{defbeta2} for $T\in\K(M_L^{\otimes i})$. For $j<i$ and $S\in \K(M_L^{\otimes j})$, we have
\begin{align*}
\beta(\iota^i(S\otimes_A1_{i-j}))&=\beta(\iota^j(S))=\iota^{j+1}(U_jSU_j^*)\\
&=\iota^{i+1}((U_j\otimes_A1_{i-j})(S\otimes_A1_{i-j})(U_j^*\otimes_A1_{i-j}))\\
&=\iota^{i+1}(U_i(S\otimes_A1_{i-j})U^*_i),
\end{align*}
and adding over $j$ gives \eqref{defbeta2}.
It now follows from Theorem~\ref{relCPalg=St} that $\kappa\times V$ is an isomorphism.

Now suppose that $\phi:A\to \L(M_L)$ has range in $\K(M_L)$. Then $A\otimes_A1_1=\phi(a)$ belongs to $K(M_L)$, and hence $C_\infty=\overline{\bigcup_{i=1}^\infty \iota^i(\K(M_L^{\otimes i}))}$. Now for $i\geq 1$ and $m=m'\otimes_A q(a)$ and $n=n'\otimes_A q(b)$ in $M_L^{\otimes i}$, we have
\begin{align*}
\kappa(\iota^1(\Theta_{m,n}))
&=k_M^{\otimes(i-1)}(m')k_M(q(a))k_M(q(b))^*k_M^{\otimes(i-1)}(m')^*\\
&=k_M^{\otimes(i-1)}(m')k_A(a)VV^*k_A(b)^*k_M^{\otimes(i-1)}(m')^*\\
&=k_M^{\otimes(i-1)}(m')k_A(a)\overline{\kappa}(\overline{\beta}(1))k_A(b)^*k_M^{\otimes(i-1)}(m')^*.
\end{align*}
Thus each $\iota^i(\K(M_L^{\otimes i}))$ is contained in the ideal generated by $\overline{\beta}(1)$, and $\overline{\beta}(1)$ is full.
\end{proof} 

The following exact sequence for the $K$-theory of $\O(M_L)$ looks a little different from the usual ones (in \cite{Pim} or \cite{Kat}, for example), because it involves the core rather than the coefficient algebra.

\begin{cor}\label{prop6term}
Let $(A, \alpha, L)$ be an Exel system such that $A$ is unital, $L$ is almost faithful on $A$, and $A$ acts by compact operators on $M_L$. Let $(C_\infty,\beta)$ be the Stacey system constructed in Corollary~\ref{CP=St}. Then there exists an exact sequence
\begin{equation}\label{6term}
\xymatrix{
K_0(C_\infty)\ar[r]^{\beta_*-\id}&K_0(C_\infty)\ar[r]^{\kappa_*}&K_0(\O(M_L))\ar[d]
\\
K_1(\O(M_L))\ar[u]&K_1(C_\infty)\ar[l]_{\kappa_*}&K_1(C_\infty).\ar[l]_{\beta_*-\id}
}
\end{equation}
\end{cor}

\begin{proof}
Since $A$ acts by compact operators on $M_L$ and $A$ is unital, the identity operator on $M_L$ is compact, and $C_\infty$ is unital. The range of $\beta$ is a full corner in $C_\infty$ by Corollary~\ref{CP=St}.  Hence Theorem~4.1 of \cite{P2} gives \eqref{6term}.
\end{proof}

\subsection{Connections with a construction of Exel}\label{connectExel}
In \cite[Theorem~6.5]{E2}, Exel shows that if $A$ is unital,  $\alpha$ is injective and unital, and  there is a faithful conditional expectation $E$ of $A$ onto $\alpha(A)$, then his crossed product $A\rtimes_{\alpha,\alpha^{-1}\circ E}\N$ is isomorphic to a Stacey crossed product $\check\AA\times_{\beta'}\N$. The $C^*$-algebra $\check\AA$ is by definition a subalgebra of the $C^*$-algebraic direct limit of a sequence of algebras of the form $\L(M_i)$ for certain Hilbert modules $M_i$. Exel's hypotheses on $\alpha$ imply that $L:=\alpha^{-1}\circ E$ is a faithful transfer operator for $\alpha$ satisfying $L(1)=1$, and hence that $\phi:A\to \L(M_L)$ is injective. Thus the Exel crossed product $A\rtimes_{\alpha,L}\N$ is the Cuntz-Pimsner algebra $\O(M_L)$, and Corollary~\ref{CP=St} gives an isomorphism of $C_\infty\times_{\beta}\N$ onto $A\rtimes_{\alpha,L}\N$. It is natural to ask whether Exel's system $(\check\AA,\beta')$ is the same as the system $(C_\infty,\beta)$ appearing in Corollary~\ref{CP=St}.

Exel's module $M_i$ is the Hilbert module over $\alpha^i(A)$ associated to the expectation $\alpha^i\circ L^i$ of $A$ onto $\alpha^i(A)$ (which he denotes by $\EE_i$), 
and hence is a completion of a copy $q_i(A)$ of $A$. By restricting the action we can view $M_i$ as a module over $\alpha^{i+1}(A)$, and this induces a linear map $j_i:M_i\to M_{i+1}$; Lemma~4.7 of \cite{E2} says that there is a homomorphism $\phi_i:\L(M_i)\to \L(M_{i+1})$ characterised by $\phi_i(T)\circ j_i=j_i\circ T$. (Exel writes the maps $j_i$ as inclusions.) Since $L^i$ is a transfer operator for $\alpha^i$, $\alpha^i\circ L^i$ extends to a self-adjoint projection $\check e_i$ in $\L(M_i)$. 
Exel's $C^*$-algebra $\check\AA$ is the $C^*$-subalgebra of $\varinjlim (\L(M_i),\phi_i)$ generated by the images of $A=\L(M_0)$ and $\{\check e_i:i\in\N\}$. Propositions 4.2 and 4.3 of \cite{E2} say that $\alpha$ and $L$ extend to isometric linear maps $\alpha_i:M_i\to M_{i+1}$ and $L_i:M_{i+1}\to M_i$, Proposition~4.6 of \cite{E2} says that the maps $\beta_i':T\mapsto \alpha_i\circ T\circ L_i$ are injective homomorphisms of $\L(M_i)$ into $\L(M_{i+1})$, and Proposition 4.10 of \cite{E2} says that they induce an endomorphism $\beta'$ of $\varinjlim (\L(M_i),\phi_i)$ which leaves $\check\AA$ invariant and satisfies $\beta'(\check e_i)=\check e_{i+1}$ for $i\geq 0$.

To compare our construction with that of \cite[\S4]{E2}, we use the maps 
\[
V_i:q(a_1)\otimes_A\cdots\otimes_Aq(a_i)\mapsto
q_i\big(a_1\alpha(a_2)\alpha^2(a_3)\cdots \alpha^{i-1}(a_i)\big);
\]
the pairs $(V_i,\alpha^i)$ then form compatible isomorphisms of $(M_L^{\otimes i},A)$ onto $(M_i,\alpha^i(A))$, and induce isomorphisms $\theta_i$ of $\L(M_L^{\otimes i})$ onto $\L(M_i)$. One quickly checks that the isometries $U_i$ of Corollary~\ref{defUi} satisfy
\begin{align*}
V_{i+1}U_iV_i^{-1}(q_i(a))&=V_{i+1}U_i(q(a)\otimes_A1\cdots\otimes_A1\\
&=V_{i+1}(q(1)\otimes_Aq(a)\otimes_A\cdots \otimes q(1))\\
&=q_{i+1}(\alpha(a))=\alpha_i(q_i(a)),
\end{align*}
and similiarly $V_iU_i^*V_{i+1}(q_{i+1}(a))=q_i(L(a))=L_i(q_{i+1}(a))$. Thus our endomorphism $\Ad U_i$ is carried into Exel's $\beta_i'$. The isomorphisms $\theta_i$ combine to give an injection of our direct limit $C_\infty=\varinjlim(C_i,\phi_i)$ into $\varinjlim (\L(M_i),\phi_i)$, and since $t^it^{*i}=\beta^i(1)$ is carried into $(\beta')^i(1)=\check e_i$, the formula \eqref{presentfpas} implies that the range of this injection is $\clsp\{a\check e_ib:a,b\in A\}$, which by \cite[Proposition~4.9]{E2} is precisely $\check\AA$. So $(C_\infty,\beta)$ is indeed isomorphic to $(\check\AA,\beta')$, and Corollary~\ref{CP=St} extends \cite[Theorem~6.5]{E2}.


\section{Graph algebras as crossed products}\label{sec:graph}
Let $E=(E^0, E^1, r,s)$ be a locally-finite directed graph with no sources or sinks. We use the conventions of \cite{CBMS}.  Briefly, we think of $E^0$ as vertices, $E^1$ as edges, and $r,s:E^1\to E^0$ as describing the range and source of an edge.  Locally finite means that $E$ is both row-finite and column-finite, so that both $r^{-1}(v)$ and $s^{-1}(v)$ are finite for every $v\in E^0$.  We write $E^*$ for the set of finite paths $\mu=\mu_1\dots\mu_n$ satisfying $s(\mu_i)=r(\mu_{i+1})$ for $1\leq i\leq n-1$, and $|\mu|$ for the length $n$ of this path.  Similarly, we write $E^n$ for the set of paths of length $n$ and  $E^\infty$ for the set of infinite paths $\eta=\eta_1\eta_2\dots$.   We equip the path space $E^\infty$ with the product topology inherited from $\prod_{n=1}^\infty E^1$, which is locally compact and Hausdorff  and has a basis  consisting of the cylinder sets $Z(\mu):=\{\eta\in E^\infty:\eta_i=\mu_i\ \text{for $1\leq i\leq |\mu|$}\}$ parametrised by $\mu\in E^*$.

Now consider the backward shift $\sigma$ on $E^\infty$ defined by $\sigma(\eta_1\eta_2\dots)=\eta_2\eta_3\dots$. Since $E$ has no sinks, $\sigma$ is surjective.  Since  $E$ is column-finite, $\sigma$ is a local homeomorphism which is proper in the sense that inverse images of compact sets are compact (see \cite[\S2.2]{BRV}).  Since $\sigma$ is proper, $\alpha:f\mapsto f\circ\sigma$ is a nondegenerate endomorphism of $C_0(E^\infty)$; since $E$ is column-finite, $\sigma^{-1}(\eta)$ is finite, and  
\[
L(f)(\eta)=\frac{1}{|\sigma^{-1}(\eta)|}\sum_{\sigma(\xi)=\eta} f(\xi)=\frac{1}{|s^{-1}(r(\eta))|}\sum_{s(e)=r(\eta)}f(e\eta)
\]
defines a transfer operator $L:C_0(E^\infty)\to C_0(E^\infty)$ for $\alpha$ \cite[Lemma~2.2]{BRV}. The same formula defines an operator $\overline{L}$ on $C_b(E^\infty)=M(C_0(E^\infty))$, so $(C_0(E^\infty), \alpha, L)$ is an Exel system. A partition of unity argument (as preceding \cite[Corollary~4.2]{BRV}) shows that the action of $C_0(E^\infty)$ on $M_L$ is by compact operators. Hence $C_0(E^\infty)\rtimes_{\alpha, L}\N:=\O(K_\alpha, M_L)=\O(M_L)$.

A Cuntz-Krieger $E$-family in a $C^*$-algebra $B$ consists of a set $\{P_v:v\in E^0\}$ of mutually orthogonal projections and a family $\{T_e:e\in E^1\}$ of partial isometries such that $T^*_eT_e=P_{s(e)}$ for all $e\in E^1$ and $P_v=\sum_{r(e)=v}T_eT_e^*$ for all $v\in E^0$.  The $C^*$-algebra $C^*(E)$ of $E$  is universal for Cuntz-Krieger $E$-families; we write $\{t,p\}$ for the universal Cuntz-Krieger $E$-family that generates $C^*(E)$.
See \cite{CBMS} for more details.

 For $e\in E^1$, we define $m_e:=|s^{-1}(s(e))|^{1/2}q(\chi_{Z(e)})$. Since $E$ is locally finite with no sources or sinks,  we know from \cite[Theorem~5.1]{BRV} that
\[
T_e:=k_M(m_e)= |s^{-1}(s(e))|^{1/2}k_M(q(\chi_{Z(e)}))\text{\ and\ }P_v:=k_A(\chi_{Z(v)})
\]
form a Cuntz-Krieger $E$-family in $\O(M_L)$, and that $\pi_{T,P}:t_e\mapsto T_e$ and $p_v\mapsto P_v$ is an isomorphism of the graph algebra $C^*(E)$ onto $C_0(E^\infty)\rtimes_{\alpha,L}\N=\O(M_L)$. Note that $\pi_{T,P}$ is equivariant for the gauge actions.  Pulling over the endomorphism $\alpha'$ of Theorem~\ref{relCPalg=St} gives an endomorphism $\beta=\pi_{T,P}^{-1}\circ\alpha'\circ \pi_{T,P}$ of the core
\[C^*(E)^\gamma=\clsp\big\{t_\mu t_\nu^*:\mu,\nu\in E^*\text{ satisfy }|\mu|=|\nu|\big\}.\] 
We are going to compute $\beta$.

Let $\mu=\mu_1\dots\mu_i$ be a finite path  and $m_\mu:=m_{\mu_1}\otimes_A\cdots\otimes_A m_{\mu_i}$. Then
\[
T_\mu=T_{\mu_1}T_{\mu_2}\cdots T_{\mu_i}=k_M(m_{\mu_1})k_M(m_{\mu_2})\cdots k_M(m_{\mu_i})=k_M^{\otimes i}(m_\mu).
\]
To compute $\beta$ we first note that
 $m_\mu=\chi_{Z(r(\mu))}\cdot m_\mu$ and \[\alpha(\chi_{Z(\mu)})=\chi_{Z(\mu)}\circ\sigma=\sum_{s(e)=\mu}\chi_{Z(e)}.\]  So for paths $\mu$ and $\nu$ of length $i$ we have
\begin{align*}
\pi_{T,P}&(\beta(t_\mu t_\nu^*))
=\alpha'(T_\mu T_\nu^*)
=\alpha'(k_M^{\otimes i}(m_\mu)k_M^{\otimes i}(m_\nu)^*)\\
&=\alpha'\big(k_M^{\otimes i}(\chi_{Z(r(\mu))}\cdot m_\mu)k_M^{\otimes i}(\chi_{Z(r(\nu))}\cdot m_\nu)^*\big)\\
&=k_M^{\otimes i+1}(q(\alpha(\chi_{Z(r(\mu))})\otimes m_\mu)k_M^{\otimes i+1}(q(\alpha(\chi_{Z(r(\nu))})\otimes m_\nu)^*\quad  \text{(using \eqref{defalpha'})}\\
&=\sum_{s(e)=r(\mu),\;s(f)=r(\nu)}(|s^{-1}(s(e))|\,|s^{-1}(s(f))|)^{-1/2}k_M^{\otimes i+1}(m_e\otimes_Am_\mu)k_M^{\otimes i+1}(m_f\otimes_Am_\nu)^*\\
&=\sum_{s(e)=r(\mu),\;s(f)=r(\nu)}(|s^{-1}(s(e))|\,|s^{-1}(s(f))|)^{-1/2}k_M^{\otimes i+1}(m_{e\mu})k_M^{\otimes i+1}(m_{f\nu})^*.
\end{align*}
Now recall that $\alpha'$ is conjugation by an isometry   $V=\lim_\lambda k_M(q(e_\lambda))$, where $\{e_\lambda\}$ is  is any approximate identity of $C_0(E^\infty)$.   A quick calculation with, for example,  the approximate identity $\{e_F=\sum_{e\in F} \chi_{Z(e)}\}$ indexed by finite subsets $F$ of $E^1$, shows that  $W:=\overline{\pi}_{T,P}^{-1}(V)=\sum_{e\in E^1}|s^{-1}(s(e))|^{-1/2}t_e$. So Theorem~\ref{relCPalg=St} gives:

\begin{prop}\label{CK=cp}
Suppose that $E$ is a locally-finite directed graph with no sources or sinks,  and $\{t_e,p_v\}$ is the universal Cuntz-Krieger $E$-family which generates $C^*(E)$. Then there is an endomorphism $\beta$ of the core $C^*(E)^\gamma$ such that
\begin{equation}\label{defbeta3}
\beta(t_\mu t_\nu^*)=\sum_{s(e)=r(\mu),\;s(f)=r(\nu)}\big(|s^{-1}(r(\mu))|\,|s^{-1}(r(\nu))|\big)^{-1/2}t_{e\mu} t_{f\nu}^*.
\end{equation}
The series
\begin{equation*}
\sum_{e\in E^1}|s^{-1}(s(e))|^{-1/2}t_e
\end{equation*}
converges strictly in $M(C^*(E))$ to an isometry $W$  satisfying $\beta(t_\mu t_\nu^*)=Wt_\mu t_\nu^* W^*$. If $\iota$ is the inclusion of the core $C^*(E)^\gamma$ in $C^*(E)$, then the associated  representation $\iota\times W$ of the Stacey crossed product $C^*(E)^\gamma\times_{\beta}\N$ is an isomorphism   onto $C^*(E)$. 
\end{prop}

Although we found the endomorphism $\beta$ using our general construction, and were surprised to find it, we have now learned that other authors have shown that Cuntz-Krieger algebras can be realised as a Stacey crossed product by an endomorphism of the core (see \cite[Example~2.5]{Ror}, \cite[Theorem~3.2]{Kwa}, and Remark~\ref{remark-kwa} below). Now that we have found our $\beta$ we should be able to prove Proposition~\ref{CK=cp} directly. Before doing this we will revisit the need for the hypotheses on the graph $E$: we used row-finiteness and the lack of sources to get that the  path space $E^\infty$ is locally compact, but a direct proof should  not go through $C_0(E^\infty)$.
Our formula for $\beta$, on the other hand, makes sense when $E$ is column-finite and has no sinks: the coefficients in \eqref{defbeta3} are crucial, as we will see in the proof of the next result. It seems likely that column-finiteness is necessary. There is no obvious way to adjust for sinks, either: if we try to interpret empty sums as $0$, then $\beta(t_\mu t_\nu^*)$ would be $0$ if either $\mu$ or $\nu$ ends at a sink, but this property is not preserved by multiplication. (If $\nu$ ends at a sink but $\mu$ doesn't, then we'd have $\beta(t_\mu t_\nu^*)=0$ but $\beta((t_\mu t_\nu^*)(t_\nu t_\mu^*))=\beta(t_\mu t_\mu^*)\not=0$.) So we settle for the following.

\begin{thm}\label{genendodecomp}
Suppose that $E$ is a column-finite directed graph with no sinks, and $\{t_e,p_v\}$ is the universal Cuntz-Krieger $E$-family which generates $C^*(E)$. Then there is an endomorphism $\beta$ of the core $C^*(E)^\gamma$ such that
\begin{equation}\label{defbeta666}
\beta(t_\mu t_\nu^*)=\sum_{s(e)=r(\mu),\;s(f)=r(\nu)}\big(|s^{-1}(r(\mu))|\,|s^{-1}(r(\nu))|\big)^{-1/2}t_{e\mu} t_{f\nu}^*.
\end{equation}
The series 
\begin{equation*}
\sum_{e\in E^1}|s^{-1}(s(e))|^{-1/2}t_e
\end{equation*}
converges strictly in $M(C^*(E))$ to an isometry $W$  satisfying $\beta(t_\mu t_\nu^*)=Wt_\mu t_\nu^* W^*$. If $\iota$ is the inclusion of the core $C^*(E)^\gamma$ in $C^*(E)$, then the associated  representation $\iota\times W$ of the Stacey crossed product $C^*(E)^\gamma\times_{\beta}\N$ is an isomorphism  onto $C^*(E)$. 
 \end{thm}

\begin{proof}
We use the notation of \cite[\S3]{CBMS} and Appendix~A. For each $v\in E^0$, $\{t_\mu t_\nu^*:|\mu|=|\nu|=i, s(\mu)=s(\nu)=v\}$ is a set of matrix units for $\F_i(v)$ (see \cite[page~312]{BPRS}). We claim their images $e_{\mu,\nu}$ under $\beta$ in $\F_{i+1}(v)$, defined by the right-hand side of \eqref{defbeta666}, are also matrix units. The product $e_{\mu,\nu}e_{\kappa,\lambda}$ contains terms like $t_{e\mu}t_{f\nu}^*t_{g\kappa}t_{h\lambda}^*$, which is zero unless $f=g$ and $\nu=\kappa$. Since we then have $r(\nu)=r(\kappa)$, the two central terms in the coefficient are the same, and
\[
e_{\mu,\nu}e_{\kappa,\lambda}=\sum_{s(e)=r(\mu),\;s(f)=r(\nu),\;s(h)=r(\lambda)}|s^{-1}(r(\mu))|^{-1/2}|s^{-1}(r(\nu))|^{-1}|s^{-1}(r(\lambda))|^{-1/2}t_{e\mu}t_{h\lambda}^*.
\]
For fixed $e$ and $h$, there are $|s^{-1}(r(\nu))|$ edges $f$  with $s(f)=r(\nu)$, and for each of these the summand is exactly the same. So 
\[
e_{\mu,\nu}e_{\kappa,\lambda}=\sum_{s(e)=r(\mu),\;s(h)=r(\lambda)}|s^{-1}(r(\mu))|^{-1/2}|s^{-1}(r(\lambda))|^{-1/2}s_{e\mu}s_{h\lambda}^*=e_{\mu,\lambda}.
\]
(Notice that the coefficients in the definition of $e_{\mu,\nu}$ had to be just right for this to work.) Thus $\{e_{\mu,\nu}\}$ is a set of matrix units as claimed, and there is a well-defined homomorphism $\beta_i:\F_i(v)\to \F_{i+1}(v)$ satisfying \eqref{defbeta666} (with $\beta$ replaced by $\beta_i$). These combine to give a homomorphism $\beta_i$ of $\F_i=\bigoplus_v \F_i(v)$ into $\F_{i+1}=\bigoplus_v\F_{i+1}(v)$.

To define $\beta$ on $C_i:=\F_0+\F_1+\cdots +\F_i$, we take the $i$-expansion $c=\sum_{j=0}^ic_j$ described in Proposition~\ref{describecore}(b), and define
$\beta^i(c)=\sum_{j=0}^i\beta_j(c_j)$. The uniqueness of the $i$-expansion implies that this gives a well-defined function $\beta^i$ on each $C_i$; to check that they give a well-defined function on $\bigcup_{i=0}^\infty C_i$, we need to check that $\beta^i(c)=\beta^{i+1}(c)$ for $c\in C_i\subset C_{i+1}$. Suppose that $c\in C_i$ has $i$-expansion $c=\sum_{j=0}^ic_j$. Then the $(i+1)$-expansion is $c=\big(\sum_{j=0}^{i-1}c_j\big)+c_i'+d$, where $c_i'\in \EE_i$ and $d\in \F_i\cap \F_{i+1}$ are uniquely determined by $c_i=c_i'+d$. Lemma~\ref{idcap} implies that if $t_\mu t_\nu^*$ belongs to $\F_i\cap \F_{i+1}$, then $v:=s(\mu)=s(\nu)$ satisfies $0<|r^{-1}(v)|<\infty$, and two applications of the Cuntz-Krieger relation at $v$ show that
\begin{align*}
\beta_{i+1}(t_\mu t_\nu^*)&=\beta_{i+1}\Big(\sum_{r(g)=v}t_{\mu g} t_{\nu g}^*\Big)\\
&=\sum_{r(g)=v,\;s(e)=r(\mu g),\;s(f)=r(\nu g)}\big(|s^{-1}(r(\mu))|\,|s^{-1}(r(\nu))|\big)^{-1/2}t_{e\mu g} t_{f\nu g}^*\\
&=\sum_{s(e)=r(\mu),\;s(f)=r(\nu)}\big(|s^{-1}(r(\mu))|\,|s^{-1}(r(\nu))|\big)^{-1/2}t_{e\mu} t_{f\nu}^*\\
&=\beta_i(t_\mu t_\nu^*).
\end{align*}
Thus $\beta_i(c_i')+\beta_{i+1}(d)=\beta_i(c_i')+\beta_i(d)=\beta_i(c_i)$. Thus
\[
\beta^{i+1}(c)=\Big(\sum_{j=0}^{i-1} \beta_i(c_j)\Big)+\beta_i(c_i')+\beta_{i+1}(d)=\Big(\sum_{j=0}^{i-1} \beta_i(c_j)\Big)+\beta_i(c_i)=\beta^i(c).
\]
  
At this stage we have a well-defined map $\beta:\bigcup_{i=0}^\infty C_i\to C^*(E)^\gamma$ satisfying \eqref{defbeta666}. This map is certainly linear, and we need to prove that it is multiplicative. We consider $t_\mu t_\nu^*\in \F_i$ and $t_\kappa t_\lambda^*\in \F_j$, and we may as well suppose $i\leq j$. Then multiplying together the two formulas for $\beta^i(t_\mu t_\nu^*)$ and $\beta^j(t_\kappa t_\lambda^*)$ gives a linear combination of things like $t_{e\mu}t_{f\nu}^*t_{g\kappa} t_{h\lambda}^*$. Because $i\leq j$, this product is $0$ unless $g\kappa=f\nu\kappa'$, and then it is $t_{e\mu\kappa'}t_{h\lambda}^*$. So the sum collapses just as it did in the first paragraph, and we obtain the formula for 
\[
\beta(t_{\mu\kappa'}t_{\lambda}^*)=\beta((t_\mu t_\nu^*)(t_\kappa t_\lambda^*)).
\]
Thus $\beta$ is multiplicative. Since it is clearly $*$-preserving, it is a $*$-homomorphism, and as such is automatically norm-decreasing on each $C_i$. Thus $\beta$ extends to an endomorphism, also called $\beta$, of $\overline{\bigcup_{i=0}^\infty C_i}=C^*(E)^\gamma$. 

Next, note that for each $v\in E^0$, the partial isometries $\{t_e:s(e)=v\}$ have the same initial projection $p_v$, and mutually orthogonal range projections $t_et_e^*$, so\footnote{It is important here that there are only finitely many summands, so  column-finiteness is crucial.} $T_v:=\sum_{s(e)=v}|s^{-1}(v)|^{-1/2}t_e$ is a partial isometry with initial projection $p_v$ and range projection $\sum_{s(e)=v=s(f)}|s^{-1}(v)|^{-1}t_et_f^*$. Now the partial isometries $\{T_v:v\in E^0\}$ have mutually orthogonal initial projections and mutually orthogonal range projections, and hence their sum converges strictly to a partial isometry $W$ with initial projection $\sum_{v\in E^0}p_v=1_{M(C^*(E))}$. In other words, $W$ is an isometry.

The covariance relation $\beta(t_\mu t_\nu^*)=Wt_\mu t_\nu^* W^*$ is easy to check, and the universal property of the Stacey crossed product $(C^*(E)^\gamma\times_{\beta}\N,i_{C^*(E)^\gamma},v)$ gives a homomorphism $\iota\times W$ such that $(\iota\times W)\circ i_{C^*(E)^\gamma}=\iota$ and $(\iota\times W)(v)=W$. This homomorphism satisfies $(\iota\times W)\circ\hat\beta_z=\gamma_z\circ(\iota\times W)$, and since $\iota$ is faithful (being an inclusion), the dual-invariant uniqueness theorem (Proposition~\ref{diut}) implies that $\iota\times W$ is injective. The range contains each $t_e=|s^{-1}(s(e))|^{1/2}\iota(t_et_e^*)W$, and hence $\iota\times W$ is surjective.
\end{proof}

\begin{remark}\label{remark-kwa}
When the graph $E$ is finite and has no sinks, the endomorphism $\beta$ has also been found by Kwa\'sniewski \cite{Kwa}. He proves in \cite[Theorem~3.2]{Kwa} that $C^*(E)$ is isomorphic to a partial-isometric crossed product $C^*(E)^\gamma\rtimes_\beta\Z$ as introduced in \cite{ABL}. He then applies general results about partial-isometric crossed products from \cite{ABL} and \cite[\S1]{Kwa} to $(C^*(E)^\gamma,\beta)$, and recovers many of the main structure theorems for graph $C^*$-algebras, as they apply to finite graphs with no sinks \cite[\S3]{Kwa}.  For such graphs $E$, the endomorphism $\beta$ is conjugation by an isometry in $C^*(E)$ and is injective. Theorem~4.15 of \cite{ABL} implies that $C^*(E)^\gamma\rtimes_\beta\Z$ is isomorphic to the Exel crossed product $C^*(E)^\gamma\rtimes_{\beta,K}\N$, where $K=\beta^{-1}\circ\Ad\bar\beta(1)$. We can then deduce from \cite[Theorem~4.7]{E} that $C^*(E)$ is isomorphic to the Stacey crossed product $C^*(E)^\gamma\times_{\beta}\N$, as in Theorem~\ref{genendodecomp}.
\end{remark}

\begin{example}\label{remCuntz}
Suppose that $E$ is the bouquet of $n$ loops on a single vertex, so that $C^*(E)$ is the Cuntz algebra $\OO_n$. Then $\beta$ is not quite the same as the endomorphism $\alpha$ in the usual description of $\O_n$ as a Stacey crossed product of its core \cite{C,S,BKR}, which is defined in terms of the infinite tensor-product decomposition $\OO_n^\gamma=\bigotimes_{k=1}^\infty M_n(\C)$ by $\alpha:\bigotimes_{k=1}^\infty a_k\mapsto e_{11}\otimes\big(\bigotimes_{k=2}^\infty a_{k-1}\big)$. But it is closely related: $p=\sum_{i,j=1}^n n^{-1}e_{ij}$ is also a rank-one projection, and 
\begin{equation}\label{defendp}
\textstyle{\beta\big(\bigotimes_{k=1}^\infty a_k\big)=p\otimes\big(\bigotimes_{k=2}^\infty a_{k-1}\big).}
\end{equation} 
If $u$ is a unitary matrix such that $ue_{11}u^*=p$, then $u\otimes 1$ is a unitary in $\bigotimes_{k=1}^\infty M_n(\C)$ such that $(u\otimes 1)\alpha(a)(u\otimes 1)^*=\beta(a)$ for all $a\in \bigotimes_{k=1}^\infty M_n(\C)$. In particular, the crossed products $\O_n^\gamma\times_\alpha\N$ and $\O_n^\gamma\times_\beta\N$ are isomorphic. This argument would work for any choice of rank-one projection $p$; ours has the advantage that no obvious choice is necessary.

One naturally asks whether it is a coincidence that our convoluted constructions  yields essentially the same endomorphism, and we will show in the next section that for every projection $p$ in $M_n(\C)$, rank-one or not, there is an endomorphism $\beta_p$ of $\O_n$ satisfying \eqref{defendp}. Particularly interesting is the endomorphism associated to the identity matrix $1_n$, which turns out to be the ``canonical endomorphism'' $\beta_1(a)\mapsto \sum_{i=1}^n s_ias_i^*$. Since $\beta_1$ is unital, the Stacey crossed product is not interesting, but there is a natural transfer operator, and hence a potentially interesting Exel crossed product. We will study this in the next section. 
\end{example}

\begin{remark}
In a graph algebra $C^*(E)$, it does not seem to be so easy to write down other endomorphisms of the core, or at least ones which are not unitarily equivalent to the $\beta$ of Theorem~\ref{genendodecomp}. The map $\beta_1:a\mapsto \sum_{e\in E^1}s_eas_e^*$, which is sometimes referred to as the ``canonical endomorphism," is not a homomorphism for most graphs\footnote{To see this, consider a pair of paths $\mu$, $\nu$ with $s(\mu)=s(\nu)$ but $r(\mu)\not=r(\nu)$. Since $s_es_\mu=0$ unless $s(e)=r(\mu)$, and similarly for $s_{\nu}^*s_e^*$, we have $\beta_1(s_\mu s_\nu^*)=0$, but $\beta_1((s_\mu s_\nu^*)(s_\mu s_\nu^*)^*)=\beta_1(s_\mu s_\mu^*)\not=0$.}.
\end{remark}


\section{Endomorphisms of the UHF core of $\O_n$}\label{sec:Cuntz}

We fix $n\geq 2$, and view the UHF algebra $A:=\UHF(n^\infty)$ as the direct limit of the tensor powers $M_n^{\otimes k}$ of $M_n(\C)$ with bonding maps $\phi_k:a_1\otimes\cdots\otimes a_k\mapsto a_1\otimes\cdots\otimes a_k\otimes 1$. We think of elements of $A$ as infinite tensors, so we write $a_1\otimes\cdots\otimes a_k\otimes 1_\infty$ for the image of $a_1\otimes\cdots\otimes a_k\in M_n^{\otimes k}$ in $A$. These elements span a dense $*$-subalgebra of $A$.

Let $\{e_{ij}:1\leq i,j\leq n\}$ be the usual matrix units in $M_n$. Then the elementary tensors $e_{\mu\nu}:=e_{\mu_1\nu_1}\otimes\cdots \otimes e_{\mu_k\nu_k}$ parametrised by multiindices $\mu,\nu\in \{1,\dots,n\}^k$ are matrix units which span $M_n^{\otimes k}$. The next lemma is standard.

\begin{lemma}\label{UHF=core}
Let $\{s_i:1\leq i\leq n\}$ be the universal Cuntz family in $\O_n$, and let $\gamma:\T\to \Aut\O_n$ be the gauge action. Then there is an isomorphism $\psi$ of $A$ onto $\O_n^\gamma$ such that
\[
\psi(e_{\mu\nu}\otimes 1_\infty)=s_\mu s_\nu^*:=s_{\mu_1}\cdots s_{\mu_n}s_{\nu_n}^*\cdots s_{\nu_1}^*.
\]
\end{lemma}

We now take $N\leq n$, and let $p$ be the diagonal matrix $1_N\oplus 0_{n-N}$ in $M_n$. Then the maps $\alpha_k:a\mapsto p\otimes a$ of $M_n^{\otimes k}$ into $M_n^{\otimes (k+1)}$ are homomorphisms which satisfy $\alpha_{k+1}\circ \phi_k=\phi_{k+1}\circ\alpha_k$, and hence induce an endomorphism $\alpha$ of $A$ such that
\begin{equation}\label{UHFaction}
\alpha(a_1\otimes\cdots\otimes a_k\otimes 1_\infty)=p\otimes a_1\otimes\cdots\otimes a_k\otimes 1_\infty.
\end{equation}
Next, we identify the corner $pM_np$ with $M_N(\C)$ in the obvious way, and let $\tr:pM_np\to \C$ be the normalised trace on $pM_np$.

\begin{lemma}\label{propL}
There is a positive linear map $L:A\to A$ such that
\begin{equation}\label{deftransfer}
L(a_1\otimes\cdots\otimes a_k\otimes 1_\infty)=\tr(pa_1p)(a_2\otimes\cdots\otimes a_k\otimes 1_\infty).
\end{equation}
Then $L$ is a transfer operator for $\alpha$, and $L$ is almost faithful.
\end{lemma}

\begin{proof}
Define $L_k:M_n^{\otimes(k+1)}\to M_n^{\otimes k}$ by $L_k(a_1\otimes\cdots\otimes a_{k+1})=\tr(pa_1p)(a_2\otimes\cdots\otimes a_{k+1})$. A calculation shows that $\phi_k\circ L_k=L_{k+1}\circ\phi_{k+1}$, so the $L_k$ combine to give a linear function $L$ on the dense subalgebra $\bigcup_{k} M_n^{\otimes k}$ of $A$ satisfying \eqref{deftransfer}. Each $L_k$ is positive and linear with $L_k(1)=1$, and hence is norm-decreasing; since the inclusions of the $M_n^{\otimes k}$ in $A$ are norm-preserving, the map $L$ is norm-decreasing on $\bigcup_{k} M_n^{\otimes k}$, and hence extends uniquely to a norm-decreasing linear map $L:A\to A$. It is positive because each $L_k$ is.

To see that $L$ is a transfer operator, we take $a,b\in M_n^{\otimes k}\subset A$ and calculate:
\begin{align*}
L(\alpha(a)b)&=L(pb_1\otimes a_1b_2\otimes\cdots\otimes a_{k-1}b_{k}\otimes a_{k}\otimes 1_\infty)\\
&=\tr(pb_1)(a_1b_2\otimes \cdots\otimes a_{k-1}b_{k}\otimes a_{k}\otimes 1_\infty)\\
&=\tr(pb_1p)(a_1b_2\otimes \cdots\otimes a_{k-1}b_{k}\otimes a_{k}\otimes 1_\infty)\\
&=(a_1\otimes \cdots\otimes a_{k}\otimes 1_\infty)\tr(pb_1p)(b_2\otimes \cdots\otimes b_{k}\otimes 1_\infty)\\
&=aL(b).
\end{align*}

Next we use Lemma~\ref{UHF=core} to view $A$ as a subalgebra of $\O_n$, set $X:=\{1,\dots,n\}^\infty$, and let $\pi_S:\O_n\to B(\ell^2(X))$ be the infinite path representation. Let $\{\delta_{x}:x\in X\}$ be the usual basis for $\ell^2(X)$. Then we claim that
\begin{equation}\label{formforL}
\big(\pi_S(L(a))\delta_x\,|\,\delta_y\big)=N^{-1}\sum_{i=1}^N\big(\pi_S(a)\delta_{ix}\,|\,\delta_{iy}\big)\quad\text{for $a\in A$ and $x,y\in X$.}
\end{equation}
It suffices by linearity and continuity to check this for $a=e_{\mu\nu}\otimes 1_\infty$. Then with $\mu=\mu_1\mu'$ and $\nu=\nu_1\nu'$, 
\begin{align*}
\big(\pi_S(L(a))\delta_x\,|\,\delta_y\big)
=\big(\pi_S(\tr(pe_{\mu_1\nu_1}p)e_{\mu'\nu'}\otimes 1_\infty)\delta_x\,|\,\delta_y\big)
\end{align*}
vanishes unless $\mu_1=\nu_1\leq N$, and if so, 
\begin{equation*}
\big(\pi_S(L(a))\delta_x\,|\,\delta_y\big)=N^{-1}\big(S_{\mu'}S_{\nu'}^*\delta_x\,|\,\delta_y\big).
\end{equation*}
The sum on the right-hand side of \eqref{formforL} is
\[
N^{-1}\sum_{i=1}^N \big(S_\mu S_\nu^*\delta_{ix}\,|\,\delta_{iy}\big)=N^{-1}\sum_{i=1}^N \big(S_\nu^*\delta_{ix}\,|\,S_\mu^*\delta_{iy}\big);
\]
the $i$th summand vanishes unless $\nu_1=i=\mu_1$, which in particular forces $\nu_1=\mu_1\leq N$; then $S_{\nu}^*\delta_{ix}=S_{\nu'}^*\delta_x$, so the right-hand side of \eqref{formforL} also reduces to $N^{-1}(S_{\mu'}S_{\nu'}^*\delta_x\,|\,\delta_y)$, and we have proved \eqref{formforL}.

To see that $L$ is almost faithful, suppose that $a\in A$ has the property that $L((ab)^*ab)=0$ for all $b\in A$. Then, with $\pi_S$ as in the previous paragraph, we have
\begin{align*}
\big(\pi_S(L((ab)^*ab))\delta_{x}\,&|\,\delta_{x}\big)=0\text{ for all $b\in A$, $x\in X$}\\
&\Longrightarrow \sum_{i=1}^N \big(\pi_S((ab)^*ab)\delta_{ix}\,|\,\delta_{ix}\big)=0\text{ for all $b\in A$, $x\in X$}\\
&\Longrightarrow \sum_{i=1}^N \big(\pi_S(ab)\delta_{ix}\,|\,\pi_S(ab)\delta_{ix}\big)=0\text{ for all $b\in A$, $x\in X$}\\
&\Longrightarrow \sum_{i=1}^N \|\pi_S(ab)\delta_{ix}\|^2=0\text{ for all $b\in A$, $x\in X$}\\
&\Longrightarrow \|\pi_S(ab)\delta_{ix}\|^2=0\text{ for all $i\leq N$, $b\in A$, $x\in X$.}
\end{align*}
We can in particular take $b=e_{ji}\otimes 1_\infty$ for any $j\leq n$, and deduce that
\[
0=\|\pi_S(ab)\delta_{ix}\|=\|\pi_S(a)S_jS_i^*\delta_{ix}\|=\|\pi_S(a)\delta_{jx}\|
\]
for all $j\leq n$ and $x\in X$; since $\{\delta_{jx}:j\leq n,\;x\in X\}=\{\delta_y:y\in X\}$, we deduce that $\pi_S(a)=0$ and $a=0$. Thus $L$ is almost faithful.
\end{proof}

Now we have an Exel system $(A,\alpha, L)$, and it is natural to ask what the Exel crossed product $A\rtimes_{\alpha,L}\N$ is. 

\begin{lemma}\label{Eijon}
The elements 
\[
\big\{E_{ij}:=N^{1/2}q(e_{ij}\otimes 1_\infty):1\leq i\leq n,\;1\leq j\leq N\big\}
\]
form an orthonormal basis for the right Hilbert $A$-module $M_L$.
\end{lemma}

\begin{proof}
The Hilbert module $(M_L)_A$ is the completion of $A$ in the seminorm defined by the inner product $\langle a\,,\,b\rangle:=L(a^*b)$. Write $a_1\in M_n$ as $\sum_{i,j=1}^n c_{ij}e_{ij}$ for some $c_{ij}\in\C$. A calculation shows that $\|(a_1-a_1p)\otimes a_2\otimes\cdots\otimes a_k\otimes 1_\infty\|^2=0$. Since also $e_{ij}p=0$ for $j>N$ we have
\begin{align*}
q(a_1\otimes\cdots\otimes a_k\otimes 1_\infty)
&=q(a_1p\otimes a_2\otimes\cdots\otimes a_k\otimes 1_\infty)\\
&=\sum_{i=1}^n\sum_{j=1}^N c_{ij} q(e_{ij}\otimes a_2\otimes\cdots\otimes a_k\otimes 1_\infty)\\
&=\sum_{i=1}^n\sum_{j=1}^N c_{ij} q(e_{ij}\otimes 1_\infty)\cdot (a_2\otimes\cdots\otimes a_k\otimes 1_\infty).
\end{align*}
Thus the $E_{ij}$ generate $M_L$ as a right Hilbert $A$-module.
To see that $\{E_{ij}\}$ is orthonormal, we compute $\langle E_{ij},E_{kl}\rangle=N\tr(pe_{ij}^*e_{kl}p)1_A$: the product $e_{ij}^*e_{kl}$ vanishes unless $i=k$, and then $\langle E_{ij},E_{il}\rangle =1_A$ when $j=l$ and $0$ otherwise.
\end{proof}

\begin{lemma}\label{fromTtopi}
Suppose $\{T_{ij}:1\leq i\leq n,\;1\leq j\leq N\}$ is a Cuntz family of isometries in a $C^*$-algebra $C$. Then there is a unital homomorphism $\pi_T:A\to C$ such that \begin{equation}\label{charreps}
T_{ij}^*\pi_T(a)T_{kl}=\pi_T(\langle E_{ij},a\cdot E_{kl}\rangle)\quad\text{for all $a\in A$ and $ij, kl$.}
\end{equation}
\end{lemma}

\begin{proof}
For $\mu,\nu\in \{1,\dots,n\}^k$ we define $T_{\mu\nu}:=T_{\mu_1\nu_1}\cdots T_{\mu_k\nu_k}$. We claim that
\[
\Big\{F_{\mu\nu}:=\sum_{\lambda\in\{1,\dots,N\}^k} T_{\mu\lambda}T_{\nu\lambda}^*:\mu,\nu\in \{1,\dots,n\}^k\Big\}
\]
is a family of nonzero matrix units in $C$ such that $\sum_{\mu\in \{1,\dots,n\}^k}F_{\mu\mu}=1$. Since $\{T_{ij}\}$ is a Cuntz $nN$-family, $\{T_{\mu\lambda}\}$ is a Cuntz $(nN)^k$-family, and hence $\sum_{\mu\in \{1,\dots,n\}^k}F_{\mu\mu}=1$. Since the $T_{\mu\lambda}$ are isometries with orthogonal ranges, the $F_{\mu\nu}$ are nonzero, and 
\[
F_{\mu\nu}F_{\sigma\tau} =\sum_{\lambda,\omega\in\{1,\dots,N\}^k} T_{\mu\lambda}T_{\nu\lambda}^*T_{\sigma\omega}T_{\tau\omega}^*
\]
is either $F_{\mu\tau}$ or $0$. 

So for each $k$, there is a homomorphism $\pi_k:M_n^{\otimes k}\to C$ such that $\pi_k(e_{\mu\nu})=F_{\mu\nu}$. For $a=e_{\mu\nu}\in M_n^{\otimes k}$, we have
\begin{align*}
\pi_{k+1}(a\otimes 1)&=\sum_{i=1}^n\pi_{k+1}(e_{\mu\nu}\otimes e_{ii})=\sum_{i=1}^n\pi_{k+1}(e_{(\mu i)(\nu i)})\\
&=\sum_{i=1}^n\sum_{\lambda\in\{1,\dots,N\}^{k+1}} T_{(\mu i)\lambda}T_{(\nu i)\lambda}^*\\
&=\sum_{\lambda'\in\{1,\dots,N\}^k}\;\sum_{i=1}^n\sum_{j=1}^N T_{(\mu i)(\lambda' j)}T_{(\nu i)(\lambda' j)}^*\\
&=\sum_{\lambda'\in\{1,\dots,N\}^k}T_{\mu\lambda'}\Big(\sum_{i=1}^n\sum_{j=1}^N T_{ij}T_{ij}^*\Big)T_{\nu \lambda'}^*\\
&=\sum_{\lambda'\in\{1,\dots,N\}^k}T_{\mu\lambda'}T_{\nu \lambda'}^*=\pi_k(e_{\mu\nu})=\pi_k(a).
\end{align*}
Thus we have a well-defined unital homomorphism $\pi_T:A\to C$ such that 
\[
\pi_T(e_{\mu\nu}\otimes 1_\infty)=\sum_{\lambda\in\{1,\dots,N\}^k} T_{\mu\lambda}T_{\nu\lambda}^*,
\]
and it remains for us to check that $(\pi_T,\{T_{ij}\})$ satisfy \eqref{charreps}. It suffices to check this on $a$ of the form $e_{\mu\nu}\otimes 1_\infty$.

So suppose $e_{\mu\nu}\in M_n^{\otimes (k+1)}$. Then
\[
T_{ij}^*\pi_T(e_{\mu\nu}\otimes1_\infty)T_{kl}
=T_{ij}^*\Big(\sum_{\lambda\in\{1,\dots,N\}^{k+1}}T_{\mu_1\lambda_1}\cdots T_{\mu_{k+1}\lambda_{k+1}}T_{\nu_{k+1}\lambda_{k+1}}^*\cdots T_{\nu_1\lambda_1}^*\Big)T_{kl}.
\]
The $\lambda$-summand vanishes unless $ij=\mu_1\lambda_1$ and $\nu_1\lambda_1=kl$, so all terms vanish unless $j=l$, $i=\mu_1$ and $k=\nu_1$. Thus if we again  write $\mu=\mu_1\mu'$, etc, then
\begin{equation}\label{LHS=}
T_{ij}^*\pi_T(e_{\mu\nu}\otimes1_\infty)T_{kl}
=\begin{cases}\sum_{\lambda'\in\{1,\dots,N\}^{k}}T_{\mu'\lambda'}T_{\nu'\lambda'}^*&\text{if $j=l$, $i=\mu_1$ and $k=\nu_1$}\\
0&\text{otherwise.} 
\end{cases}
\end{equation}
On the other hand, we have
\begin{align*}
\pi_T\big(\langle E_{ij},e_{\mu\nu}\cdot E_{kl}\rangle\big)
&=N^{1/2}\pi_T\big(\langle E_{ij},(e_{\mu_1\nu_1}e_{kl})\otimes e_{\mu'\nu'}\otimes 1_\infty\rangle\big)\\
&=\begin{cases}0&\text{unless $k=\nu_1$}\\
N^{1/2}\pi_T\big(\langle E_{ij},e_{\mu_1 l}\otimes e_{\mu'\nu'}\otimes 1_\infty\rangle\big)&\text{if $k=\nu_1$}
\end{cases}\\
&=\begin{cases}0&\text{unless $k=\nu_1$}\\
N\pi_T\big(\tr(pe_{ij}^*e_{\mu_1l}p)(e_{\mu'\nu'}\otimes 1_\infty)\big)&\text{if $k=\nu_1$}
\end{cases}\\
&=\begin{cases}0&\text{unless $k=\nu_1$, $i=\mu_1$ and $j=l$}\\
N\pi_T\big(N^{-1}(e_{\mu'\nu'}\otimes 1_\infty)\big)&\text{if $k=\nu_1$, $i=\mu_1$ and $j=l$,}
\end{cases}
\end{align*}
which is the same as \eqref{LHS=}.
\end{proof}

To finish our analysis of $A\rtimes_{\alpha, L}\N$ we need a general lemma.  First, recall that a finite set $\{F_i:1\leq i\leq k\}$ in a right Hilbert $B$-module $M$ is a \emph{Parseval frame} if
\begin{equation}\label{defParseval}
m=\sum_{i=1}^k F_i\cdot\langle F_i, m\rangle_B\ \text{ for every $m\in M$}.
\end{equation}

\begin{lemma}\label{lem-universal}
Let $(B,\beta,K)$ be an Exel system with $B$ unital. Suppose that $\{F_i\}_{i=1}^k$ is a Parseval frame for  $M_K$, that $\pi:B\to C$ is a unital homomorphism, and that $\{S_{i}:1\leq i\leq k\}$ is a Cuntz family of isometries in $C$ such that
\begin{equation}\label{conditiononpi} 
S_{i}^*\pi(b)S_{j}=\pi(\langle F_{i},b\cdot F_{j}\rangle_K)\quad\text{for all $b\in B$ and $1\leq i, j\leq k$.}
\end{equation}
Define $\psi:M_K\to C$ by 
\[
\psi(m)=\sum_{i=1}^kS_{i}\pi(\langle F_{i},m\rangle_K).
\]
Then $(\psi,\pi)$ is a Cuntz-Pimsner covariant representation of $M_K$ in $C$ such that $\psi(F_i)=S_i$ for all $i$.
\end{lemma}

\begin{proof}
Let $m,n\in M_K$ and $b\in B$. We compute:
\[
\psi(m\cdot b)=\sum_{i=1}^k S_i\pi(\langle F_i\,,\, m\cdot b\rangle_L)=\sum_{i=1}^k S_i\pi(\langle F_i\,,\, m\rangle_Kb)=\psi(m)\pi(b),
\]
\begin{align*}
\psi(m)^*\psi(n)&=\sum_{i,j=1}^k \pi(\langle F_i\,,\, m\rangle_K)^*S_i^*S_j\pi(\langle F_j\,,\, n\rangle_K)\\
&=\sum_{i,j=1}^k \pi(\langle F_i\,,\, m\rangle_K)^*\pi(\langle F_i\,,\,F_j\rangle_K)\pi(\langle F_j\,,\, n\rangle_K)\\
&=\sum_{i,j=1}^k\pi\big(\big\langle F_i\cdot\langle F_i\,,\, m\rangle_K\,,\,F_j\cdot \langle F_j\,,\, n\rangle_K\big\rangle_K\big)\\
&=\pi(\langle m\,,\, n\rangle_K),
\end{align*}
and 
\begin{align*}
\psi(b\cdot m)&=\sum_{i=1}^k S_i\pi(\langle F_i\,,\, b\cdot m\rangle_K)
=\sum_{i=1}^k S_i\pi\Big( \Big\langle F_i\,,\, b\cdot\sum_{j=1}^n F_j\cdot\langle F_j\,,\, m\rangle_K\Big\rangle_K \Big)\\
&=\sum_{i,j=1}^k S_i\pi(\langle F_i\,,\, b\cdot F_j\rangle_K\langle F_j\,,\, m\rangle_K)
=\sum_{i,j=1}^k S_iS_i^*\pi(b)S_j\pi(\langle F_j\,,\, m\rangle_K)\\
&=\Big(\sum_{i=1}^kS_iS_i^*\Big)\Big(\sum_{j=1}^n\pi(b)S_j\pi(\langle F_j\,,\, m\rangle_K)\Big)=\pi(b)\psi(m).
\end{align*}
Thus $(\psi,\pi)$ is a  representation of $M_K$. Next we check the Cuntz-Pimsner covariance of $(\psi,\pi)$. Using the reconstruction formula \eqref{defParseval} we see that
\begin{equation*}\label{LHactionFR}
\phi(b)m=\sum_{i=1}^k b\cdot(F_i\cdot\langle F_i\,,\, m\rangle_K)=\sum_{i=1}^k (b\cdot F_i)\cdot \langle F_i\,,\, m\rangle_K)=\sum_{i=1}^k\Theta_{b\cdot F_i,F_i}(m).
\end{equation*}
Thus
\begin{align*}
(\psi,\pi)^{(1)}(\phi(b))
&=(\psi,\pi)^{(1)}\Big(\sum_{i=1}^k  \Theta_{b\cdot F_i, F_i}\Big)
=\sum_{i=1}^k \psi(b\cdot F_i)\psi(F_i)^*\\
&=\sum_{i=1}^k\pi(b)\psi(F_i)\psi(F_i)^*=\pi(b)\Big(\sum_{i=1}^k S_iS_i^*\Big)=\pi(b),
\end{align*}
and so $(\psi,\pi)$ is Cuntz-Pimsner covariant. Using \eqref{conditiononpi} we have 
\[
\psi(F_i)=\sum_{i=1}^k S_j\pi(\langle F_j\,,\, F_i\rangle_K)=\sum_{i=1}^k S_jS_j^*\pi(1)S_i=S_i
\]
since $\pi$ is unital.
\end{proof}

\begin{thm}\label{ExelcpOn} Let $A:=\UHF(n^\infty)$ and $(A,\alpha, L)$ the Exel system described by \eqref{UHFaction}--\eqref{deftransfer}.
Then there is an isomorphism $\psi\times\pi$ of $A\rtimes_{\alpha,L}\N=\O(M_L)$ onto $\O_{nN}$.
\end{thm}

\begin{proof}
We parametrize the set $\{1,\dots,nN\}$ by $\{ij:1\leq i\leq n,\;1\leq j\leq N\}$, and apply Lemma~\ref{fromTtopi} to the canonical Cuntz family $s=\{s_{ij}\}$  in $\O_{nN}$. This gives a unital homomorphism $\pi_s:A\to \O_{nN}$ such that \eqref{charreps} holds. The orthonormal basis $\{E_{ij}\}$ of Lemma~\ref{Eijon}  is in particular a Parseval frame, and so  Lemma~\ref{lem-universal}   gives a Cuntz-Pimsner covariant representation $(\psi,\pi_s)$ of $M_L$ in $\O_{nN}$ such that $\psi(E_{ij})=s_{ij}$. To see that $\psi\times\pi_s$ is injective, we verify the hypotheses of \cite[Corollary~5.3]{BR}. Since $A$ is simple and $\pi_s$ is nonzero, $\pi_s$ is injective. The reconstruction formula for  $\{E_{ij}\}$ implies that $\phi(a)=\sum_{i,j} \Theta_{a\cdot E_{ij},E_{ij}}$, and in particular that $\phi(A)\subset \K(M_L)$. Since $A$ is simple, the ideal $\overline {A\alpha(A)A}$ is all of $A$, and this implies that $K_\alpha=A=\phi^{-1}(\K(M_L))$. Lemma~\ref{propL} says that the transfer operator $L$ is almost faithful. The gauge action $\gamma$ on $\O_{nN}$ satisfies
\begin{equation}\label{inv}
\gamma_z(\psi(E_{ij}))=\gamma_z(s_{ij})=zs_{ij}=\psi(zE_{ij})=\psi(\widehat\alpha_z(E_{ij})).
\end{equation}
Thus \cite[Corollary~5.3]{BR} implies that $\psi\times \pi_s$ is an isomorphism of $A\rtimes_{\alpha, L}\N$ into $\O_{nN}$. Since the range of $\psi\times \pi_T$ contains every generator $s_{ij}=\pi_s(E_{ij})$, $\psi\times \pi_s$ is also surjective.
\end{proof}

\begin{remark}
Since the isomorphism of Theorem~\ref{ExelcpOn} intertwines the dual action $\widehat\alpha$ and the gauge action $\gamma$ on $\O_{nN}$ (see \eqref{inv}), it carries the fixed-point algebra $(A\rtimes_{\alpha,L}\N)^{\widehat\alpha}$ onto $\O_{nN}^\gamma=\UHF((nN)^\infty)$. So Corollary~\ref{Ex=St2} gives a description of $A\rtimes_{\alpha,L}\N$ as a Stacey crossed product by a corner endomorphism on $\UHF((nN)^\infty)$.
\end{remark}

The rank $N$ of the projection $p\in M_n(\C)$ in Theorem~\ref{ExelcpOn} is constrained by $1\leq N\leq n$. However, it is relatively easy to remove this constraint.

\begin{cor}
Let $m,k\in\N$, let $p\in M_m^{\otimes k}$ be a projection of rank $N$, and let $\tr$ be the normalised trace on $pM_m^{\otimes k}p$. Then there are an endomorphism $\alpha$ of $\UHF(m^\infty)=\varinjlim_j M_m^{\otimes kj}$such that
\[
\alpha(a_1\otimes\cdots\otimes a_j\otimes 1_\infty)=p\otimes a_1\otimes\cdots\otimes a_j\otimes 1_\infty,
\]
and a transfer operator $L$ for $\alpha$ such that 
\[
L(a_1\otimes\cdots\otimes a_j\otimes 1_\infty)=\tr(pa_1p)(a_2\otimes\cdots\otimes a_j\otimes 1_\infty).
\]
Then $\UHF(m^\infty)\rtimes_{\alpha, L}\N$ is isomorphic to $\O_{m^kN}$.
\end{cor}

\begin{proof}
The results of this section apply with $n=m^k$ to give a system  $(\UHF(n^\infty),\alpha, L)$ with $\UHF(n^\infty)\rtimes_{\alpha, L}\N$ isomorphic to $\O_{m^kN}$. Choosing an isomorphism of $M_m^{\otimes k}$ with $M_n$ gives an isomorphism of $\UHF(m^\infty)=\varinjlim_j M_m^{\otimes kj}$ onto $\UHF(n^\infty)=\varinjlim_j M_n^{\otimes j}$, and pulling $\alpha$ and $L$ over to $\UHF(m^\infty)$ gives an Exel system $(\UHF(m^\infty),\alpha, L)$ with the required properties.
\end{proof}


\appendix

\section{The core in a graph algebra}\label{analysecore}

Suppose that $E$ is an arbitrary directed graph, which could have infinite receivers, infinite emitters, sources and/or sinks.  In Theorem~\ref{genendodecomp}, we wanted a description of the core $C^*(E)^\gamma$ which did not depend on row-finiteness, and since we cannot recall seeing a suitable description in the literature, we give one here. As in the row-finite case, which is done in \cite{BPRS} and \cite{CBMS}, our description uses the subspaces
\[
\F_i:=\clsp\{ t_\mu t_\nu^*:|\mu|=|\nu|=i,\ s(\mu)=s(\nu)\},
\]
which one can easily check are in fact $C^*$-subalgebras. (By convention, $\F_0=\clsp\{p_v:v\in E^0\}$.)

We refer to vertices which are either infinite receivers or sources as ``singular vertices''. Recall that no Cuntz-Krieger relation is imposed at a singular vertex $v$, but if $v$ is an infinite receiver then we impose the inequality $p_v\geq \sum_{e\in F}t_et_e^*$ for every finite subset $F$ of $r^{-1}(v)$.

\begin{prop}\label{describecore}
Let $E$ be a directed graph, and define $\F_i$ as above.

\smallskip
\textnormal{(a)} For $i\geq 0$, $C_i:=\F_0+\F_1+\cdots+\F_i$ is a $C^*$-subalgebra of $C^*(E)^\gamma$, $C_i\subset C_{i+1}$ and $C^*(E)^\gamma=\overline{\bigcup_{i=0}^\infty C_i}$.

\smallskip
\textnormal{(b)} For each $i\geq 0$ and each $c\in C_i$, there are unique elements
\[
c_j\in \EE_j:=\clsp\{t_\mu t_\nu^*:|\mu|=|\nu|=j \text{ and }\ s(\mu)=s(\nu)\text{ is singular}\}
\]
for $0\leq j<i$ and $c_i\in\F_i$ such that $c=\sum_{j=0}^i c_j$.
\end{prop}

Since $C_i\subset C_j$ for $i<j$, an element of $C_i$ has lots of the expansions described in (b). We refer to the one obtained by viewing $c$ as an element of $C_j$ as the \emph{$j$-expansion} of~$c$. 

\begin{proof}[Proof of Proposition~\ref{describecore}\textnormal{(a)}]
Since $\bigcup_i C_i$ is a vector space containing every element of the form $t_\mu t_\nu^*$ with $|\mu|=|\nu|$, it is dense in $C^*(E)^\gamma$, and we trivially have $C_i\subset C_{i+1}$. We prove that $C_i$ is a $C^*$-subalgebra by induction on $i$. For $i=0$, it is straightforward to see that $C_0=\F_0$ is a $C^*$-subalgebra. Suppose that $C_i$ is a $C^*$-subalgebra. If $|\kappa|=|\lambda|=i+1$ and $|\mu|=|\nu|\leq i+1$, then the formula
\[
(t_\mu t_\nu^*)(t_\kappa t_\lambda^*)=
\begin{cases}
0&\text{ unless $\kappa$ has the form $\nu\kappa'$}\\
t_{\mu\kappa'}t_\lambda^*&\text{if $\kappa=\nu\kappa'$}
\end{cases}
\]
shows that $\F_{i+1}$ is an ideal in the $C^*$-subalgebra $C^*(C_{i+1})$ of $C^*(E)^\gamma$ generated by $C_{i+1}$. Since $C_i$ is a $C^*$-subalgebra of $C^*(C_{i+1})$, the sum $C_{i+1}=C_i+\F_{i+1}$ is also a $C^*$-subalgebra (and in fact $C^*(C_{i+1})=C_i+\F_{i+1}=C_{i+1}$).
\end{proof}

For part (b), we need to do some preparation. The standard argument of \cite[\S2]{BPRS} or \cite[Chapter~3]{CBMS} shows that, for fixed $i$ and $v\in E^0$, 
\[
\{t_\mu t_\nu^*:|\mu|=|\nu|=i,\ s(\mu)=s(\nu)=v\}
\]
is a set of non-zero matrix units, and hence their closed span $\F_i(v)$ is a $C^*$-subalgebra of $\F_i$ isomorphic to $\K(\ell^2(E^i\cap s^{-1}(v)))$. Since $\F_i(v)\F_i(w)=\{0\}$ for $v\not=w$, $\F_i$ is the $C^*$-algebraic direct sum $\bigoplus_{v\in E^0}\F_i(v)$. For $j$ satisfying $0\leq j\leq i$, we set
\[
D_{j,i}:=\F_j+\cdots+\F_i=\clsp\{ t_\mu t_\nu^*:j\leq |\mu|=|\nu|\leq i,\ s(\mu)=s(\nu)\},
\]
which is another $C^*$-subalgebra by the argument in the previous proof. It is an ideal in $C_i$. Now we prove a lemma.

\begin{lemma}\label{idcap}
For every $i\geq 1$ and every $j<i$, we have
\[
\F_j(v)\cap D_{j+1,i}=
\begin{cases}
0&\text{if $v$ is a singular vertex}\\
\F_j(v)&\text{if $0<|r^{-1}(v)|<\infty$.}
\end{cases}
\]
\end{lemma}

\begin{proof}
The result is trivially true if $\F_j(v)$ is $\{0\}$, so we suppose it isn't. Since $\F_j(v)\cap D_{j+1,i}$ is an ideal in $\F_j(v)$, it is either $\{0\}$ or $\F_j(v)$.    
First suppose that $v$ is not a singular vertex. Then the Cuntz-Krieger relation at $v$ implies that $\F_j(v)\subset D_{j+1,i}$, and hence $\F_j(v)\cap D_{j+1,i}=\F_j(v)$.

Now suppose that $\F_j(v)\cap D_{j+1,i}=\F_j(v)$; we will show that $v$ is not singular. Choose $\mu\in $$E^j$ with $s(\mu)=v$. Then $t_\mu t_\mu^*$ is a non-zero element of $D_{j+1,i}$, so there exist $\kappa$ and $\lambda$ satisfying $j+1\leq|\kappa|=|\lambda|\leq i$ and $(t_\mu t_\mu^*)(t_\kappa t_\lambda^*)\not= 0$. But this implies that $\kappa$ has the form $\mu\kappa'$, and $v=s(\mu)$ cannot be a source. It also implies that $t_\mu t_\mu^*t_\kappa t_\kappa^*$ is non-zero, and hence so is the larger projection $t_\mu t_\mu^* t_{\mu\kappa_{j+1}}t_{\mu\kappa_{j+1}}^*$. Thus $\F_j(v)\cap \F_{j+1}\neq \{0\}$, and since $\F_j(v)\cap \F_{j+1}$ is  an  ideal in $\F_j(v)$, it follows that $\F_j(v)\cap \F_{j+1}=\F_j(v)$.

We now know that the projection $t_\mu t_\mu^*$ belongs to $\F_{j+1}$, and hence is a projection in a $C^*$-algebraic direct sum $\bigoplus_{w\in E^0}\F_{j+1}(w)$. The norms of elements in this direct sum are arbitrarily small off finite subsets of $E^0$, and projections have norm $0$ or $1$, so there are a finite subset $F$ of $E^0$ and projections $q_w\in \F_{j+1}(w)$ such that $t_\mu t_\mu^*=\sum_{w\in F}q_w$. Each $q_w$ is a projection in $\F_{j+1}(w)=\K(\ell^2(E^{j+1}\cap s^{-1}(w)))$, and hence has finite trace. Thus $t_\mu t_\mu^*$ has finite trace. On the other hand, for every edge $e$ with $r(e)=s(\mu)=v$, we have $t_\mu t_\mu^*\geq t_{\mu e}t_{\mu e}^*$; since $\{t_{\mu e}t_{\mu e}^*:r(e)=s(\mu)\}$ is a family of mutually orthogonal projections of trace $1$, we have
\[
\Tr (t_\mu t_\mu^*)\geq \sum_{r(e)=s(\mu)}\Tr(t_{\mu e}t_{\mu e}^*)=|r^{-1}(s(\mu))|=|r^{-1}(v)|.
\]
Since we have already eliminated the possibility that $v$ is a source, this proves that it is not a singular vertex, as required.
\end{proof}

\begin{proof}[Proof of Proposition~\ref{describecore}\textnormal{(b)}]
Lemma~\ref{idcap} implies that
\[
\F_j\cap D_{j+1,i}=\bigoplus\{\F_j(v):0<|r^{-1}(v)|<\infty\},
\]
and that $\F_j$ is the direct sum of $\F_j\cap D_{j+1,i}$ and 
\[
\EE_j=\bigoplus\{\F_j(v):\text{$v$ is a singular vertex}\}.
\]
Since $D_{j,i}=\F_j+D_{j+1,i}$, we have
\begin{align*}
C_i&=\F_0+D_{1,i}=(\EE_0+(\F_0\cap D_{1,i}))+D_{1,i}=\EE_0+D_{1,i}\\
&=\EE_0+(\EE_1+(\F_1\cap D_{2,i}))
=\EE_0+\EE_1+D_{2,i}\\
&\ \vdots \\
&=\EE_0+\cdots \EE_{i-1}+D_{i,i}=\EE_0+\cdots \EE_{i-1}+\F_i,
\end{align*}
which shows that $c$ has the claimed expansion.

To establish uniqueness, suppose that $c_j,d_j\in \EE_j$ for $j<i$, that $c_i, d_i\in \F_i$, and that $\sum_{j=0}^ic_j=\sum_{j=0}^id_j$. Then $c_0-d_0=\sum_{j=1}^i(d_j-c_j)$ belongs to $\EE_0\cap \F_0$, because the left-hand side does, and to $D_{1,i}$, because the right-hand side does. Since $\F_0=\EE_0\oplus(\F_0\cap D_{1,i})$, we have $\EE_0\cap (\F_0\cap D_{1,i})=\{0\}$, and we deduce that $c_0=d_0$ and $\sum_{j=1}^ic_j=\sum_{j=1}^id_j$. Now an induction argument using $\EE_j\cap (F_j\cap D_{j+1,i})=\{0\}$ gives the result. 
\end{proof}


\end{document}